\documentclass[11pt,a4paper]{article}
\usepackage{amsmath}
\usepackage{amsfonts}
\usepackage{amssymb}
\usepackage{amsthm}
\usepackage{mathrsfs}
\usepackage{dsfont}
\usepackage{paralist}
\usepackage{natbib}
\usepackage{bm}
\usepackage{color}
\usepackage[colorlinks=true,linkcolor=blue,citecolor=blue,pdfborder={0 0 0}]{hyperref}
\usepackage{graphicx}

\newcount\Comments  
\newcommand{\kibitz}[2]{\ifnum\Comments=1\textcolor{#1}{#2}\fi}

\newtheoremstyle{normal}
{2ex}               
{3ex}               
{}                  
{}                  
{\bfseries} 
{}                  
{2pt}   
{\thmname{#1}\thmnumber{ #2.} \thmnote{(#3)}}

\newtheoremstyle{italic}
{2ex}
{3ex}
{\itshape}
{}
{\bfseries} 
{}
{2pt}
{\thmname{#1}\thmnumber{ #2.} \thmnote{(#3)}}

\theoremstyle{normal}
\newtheorem{definition}{Definition}[section]
\newtheorem{remark}[definition]{Remark}
\newtheorem{example}[definition]{Example}
\newtheorem{condition}[definition]{Condition}

\theoremstyle{italic}
\newtheorem{theorem}[definition]{Theorem}
\newtheorem{lemma}[definition]{Lemma}
\newtheorem{prop}[definition]{Proposition}
\newtheorem{corollary}[definition]{Corollary}

\newcommand\N{\mathbb{N}}

\newcommand\R{\mathbb{R}}
\newcommand\Z{\mathbb{Z}}
\newcommand\eps{\varepsilon}
\newcommand{\vect}{\bm }  
\newcommand\Prob{\mathbb{P}}    
\newcommand\Exp{\mathbb{E}}     
\newcommand\ind{\mathds{1}}     
\newcommand\Ec{\mathcal{E}}
\newcommand\Bc{\mathcal{B}}
\newcommand\Dc{\mathcal{D}}

\newcommand\Ab{\mathbb{A}}
\newcommand\Bb{\mathbb{B}}
\newcommand\Cb{\mathbb{C}}

\newcommand\Fb{\mathbb{F}}

\DeclareMathOperator{\Cov}{Cov}

\DeclareMathOperator{\const}{const}
\DeclareMathOperator{\cum}{cum}

\newcommand{\ec}{C_n}

\newcommand\weak{\rightsquigarrow}

\newcommand\bbb{ {c} }
\newcommand\dimi{ {j} }
\newcommand\udimi{ {u^{(\dimi)} } }

\allowdisplaybreaks[1]

\begin{document}

\title{Weak convergence of the empirical copula process with respect to weighted metrics}

\author{Betina Berghaus, Axel B\"ucher and Stanislav Volgushev\footnote{Ruhr-Universit\"at Bochum,
Fakult\"at f\"ur Mathematik, 
Universit\"atsstr.~150, 44780 Bochum, Germany. 
{E-mail:} betina.berghaus@rub.de, axel.buecher@rub.de, stanislav.volgushev@rub.de. 
This work has been supported by the Collaborative Research Center ``Statistical modeling of nonlinear dynamic processes'' (SFB 823) of the German Research Foundation (DFG) which is gratefully acknowledged.} \smallskip\\ 
}

\maketitle

\begin{abstract}
The empirical copula process plays a central role in the asymptotic analysis of many statistical procedures which are based on copulas or ranks. Among other applications, results regarding its weak convergence can be used to develop asymptotic theory for estimators of dependence measures or copula densities, they allow to derive tests for stochastic independence or specific copula structures, or they may serve as a fundamental tool for the analysis of multivariate rank statistics.
In the present paper, we establish weak convergence of the empirical copula process (for observations that are allowed to be serially dependent) with respect to weighted supremum distances. The usefulness of our results is illustrated by applications to general bivariate rank statistics and to estimation procedures for the Pickands dependence function arising in multivariate extreme-value theory.
\end{abstract}



\smallskip




\section{Introduction}
\def\theequation{1.\arabic{equation}}
\setcounter{equation}{0}

The theory of weak convergence of empirical processes can be regarded as one of the most powerful tools in mathematical statistics. Through the continuous mapping theorem or the functional delta method, it greatly facilitates the development of asymptotic theory in a vast variety of situations \citep{VanWel96}.

For applying the continuous mapping theorem or the functional delta method, the course of action is often similar. Consider for instance the continuous mapping theorem: starting from some abstract weak convergence result, say $\Fb_n \weak \Fb$ in some metric space $(\Dc,d_\Dc)$, one would like to deduce  weak convergence of $\phi(\Fb_n) \weak \phi(\Fb)$, where $\phi$ is some mapping defined on $(\Dc,d_\Dc)$ with values in another metric space $(\Ec,d_\Ec)$.  This conclusion is possible provided $\phi$ is continuous at every point of a set which contains the limit $\Fb$, almost surely \citep{VanWel96}. 

The continuity of $\phi$ is linked to the strength of the metric $d_\Dc$ -- a stronger metric will make more functions continuous. For example, let $\Dc=\ell^\infty([0,1])$ denote the space of bounded functions on $[0,1]$ and consider the real-valued functional $\phi(f) := \int_{(0,1)} f(x)/x \, dx$ (with $\phi$ defined on a suitable subspace of~$\Dc$). In Section~\ref{subsec:pick} below, this functional will turn out to be of great interest for the estimation of Pickands dependence function and it is also closely related to the classical Anderson-Darling statistic. Now, if we equip~$\Dc$ with the supremum distance, as is typically done in empirical process theory, the map $\phi$ is not continuous because $1/x$ is not integrable. Continuity of $\phi$ can be ensured by considering a weighted distance, such as for instance $\sup_{x\in [0,1]} |f_1(x)-f_2(x)|/g(x)$ for a positive weight function $g$ such that $ g(x)/x$ is integrable. Similar phenomenas arise with the functional delta method, see \cite{BeuZah10}. It thus is desirable to establish weak convergence results with the metric $d_\Dc$ taken as strong as possible. One class of metrics which is of particular interest in many statistical applications is given by weighted supremum distances.

For classical empirical processes, corresponding weak convergence results are well known. For example, 
the standard $d$-dimen\-sional empirical process $\Fb_n(\vect x)=\sqrt n \{ F_n(\vect x) - F(\vect x) \}$ with $F$ having standard uniform marginals, 
converges weakly with respect to the metric induced by the weighted norm 
\[
\| G \|_{\omega} = \sup_{\vect u \in [0,1]^d} \left| \frac{G(\vect u)}{ \{ g(\vect u) \}^\omega} \right|, \quad g(\vect u)= \big(\min_{j=1}^d u_j \big) \wedge \big(1-  \min_{j=1}^d u_j \big),
\]
$\omega\in(0,1/2)$. 
See, e.g., \cite{ShoWel86} for the one-dimensional i.i.d.-case, \cite{ShaYu96} for the one-dimensional time series case or \cite{GenSeg09} for the bivariate i.i.d.-case.
For $d=2$, the graph of the function $g$ is depicted in Figure~\ref{fig:level}.

\begin{figure}[h!]
\vspace{-1cm} 
\hspace{-1cm}
\includegraphics[width=0.62\textwidth]{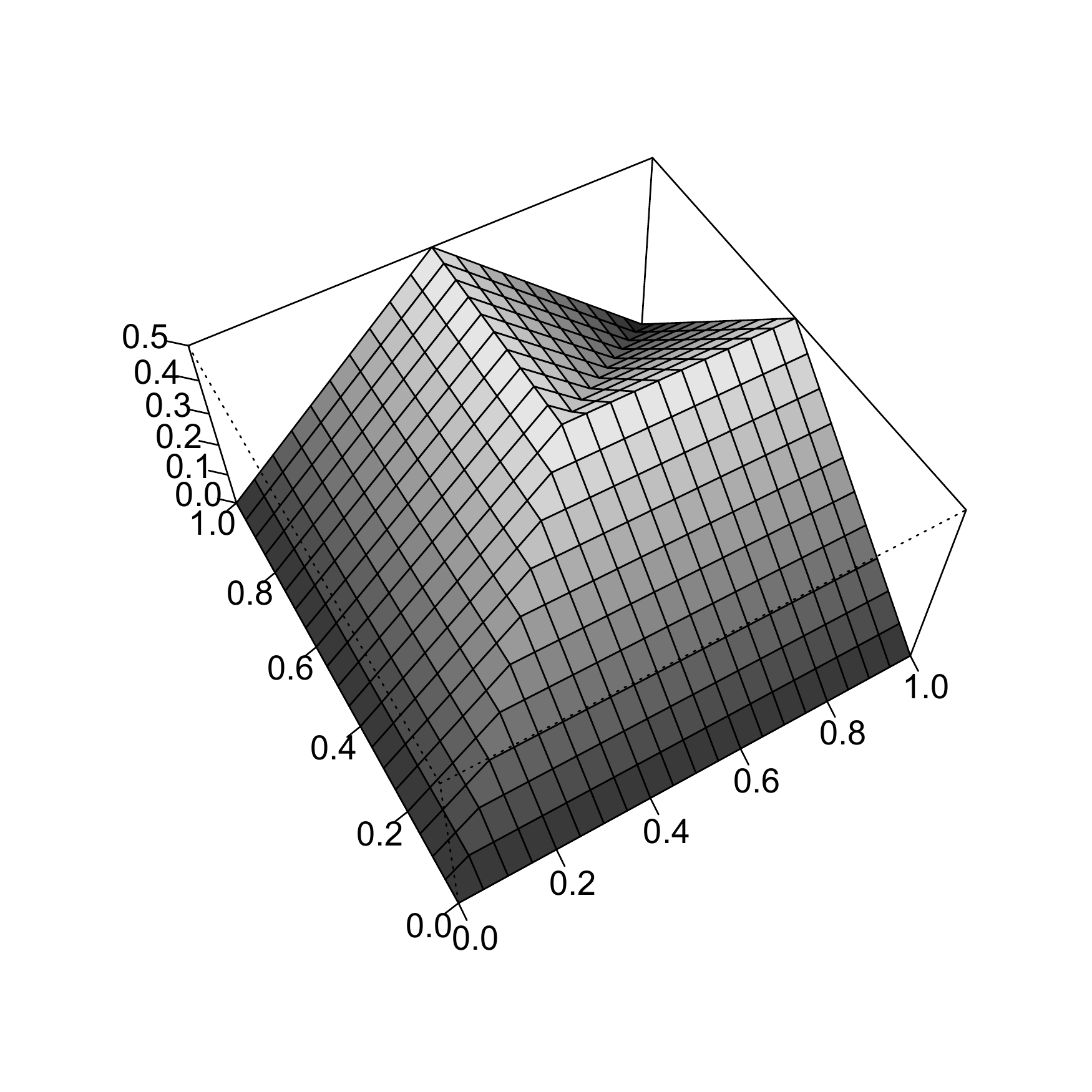} \hspace{-1.9cm}
\includegraphics[width=0.62\textwidth]{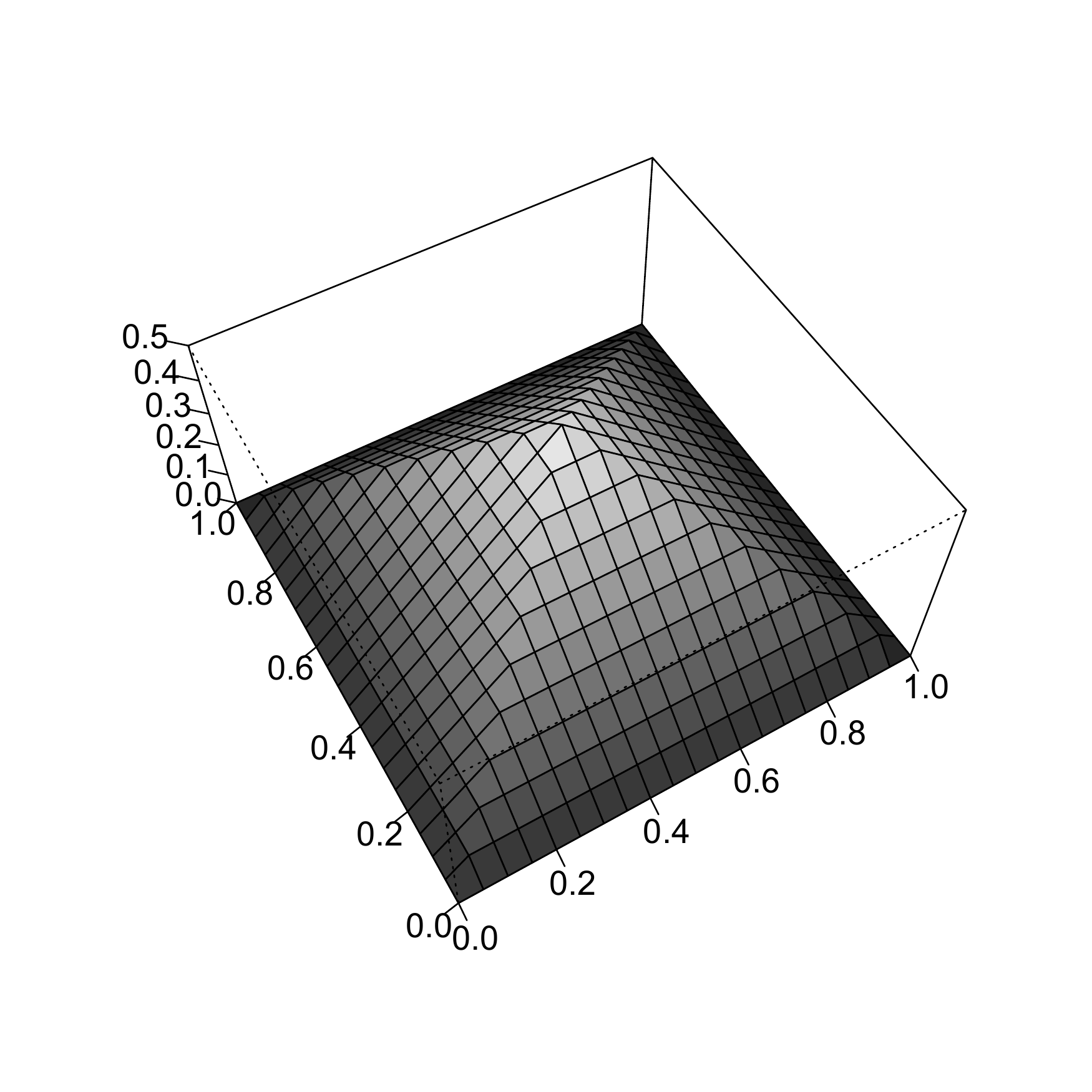}
\vspace{-1.7cm}  
\caption{\emph{Graphs of $g(u,v) = \min\{ u,v,1-\min(u,v)\}$ 
   (left picture) and of $\tilde g(u,v) = \min\{u ,v , (1-u),(1-v)\}$ (right picture).
   }} \label{fig:level}
   \vspace{-.3cm}
\end{figure}

The present paper is motivated by the apparent lack of such results for the empirical copula process $\hat \Cb_n$.
This process, an element of $D([0,1]^d)$ precisely defined in Section~\ref{sec:weighted} below, plays a crucial role in the asymptotic analysis of statistical procedures which are based on copulas or ranks. 
Unweighted weak convergence of $\hat \Cb_n$ has been investigated by several authors under a variety of assumptions on the smoothness of the copula and on the temporal dependence of the underlying observations, see \cite{GanStu87, FerRadWeg04,Seg12,BucVol13, BucSegVol14}, among others. However, results regarding its weighted weak convergence are almost non-existent. To the best of our knowledge, the only reference appears to be \cite{Rus76}, where, however, weight functions are only allowed to approach zero at the lower boundary of the unit cube. The restrictiveness of this condition becomes particularly visible in dimension $d=2$ where it is known that the limit of the empirical copula process is zero on the entire boundary of the unit square \citep{GenSeg10}. This observation suggests that, for $d=2$, it should be possible to maintain weak convergence of the empirical copula process when dividing by functions of the form $\{\tilde g(u,v)\}^\omega$ where 
\[
\tilde g(u,v) = u\wedge v \wedge (1-u) \wedge (1-v).
\] 
A picture of the graph of $\tilde g$ can be found in Figure~\ref{fig:level}, obviously, we have $\tilde g \le g$. The main result of this paper confirms the last-mentioned conjecture. More precisely, we establish weighted weak convergence of the empirical copula process in general dimension $d\ge 2$ with weight functions that approach zero wherever the potential limit approaches zero. We also do not require the observations to be i.i.d. and allow for exponential alpha mixing. 

Potential applications of the new weighted weak convergence results are extensive. As a direct corollary, one can derive the asymptotic behavior of Anderson-Darling type goodness-of-fit statistics for copulas. The derivation of the asymptotic behavior of rank-based estimators for the Pickands dependence functions \citep{GenSeg09} can be greatly simplified and, moreover, can be simply extended to time series observations. Through a suitable partial integration formula, the results can also be exploited to derive weak convergence of multivariate rank statistics as for instance of certain scalar measures of (serial) dependence. The latter two applications are worked out in detail in Section~\ref{sec:app} of this paper.

The remaining part of this paper is organized as follows. In Section~\ref{sec:weighted}, the empirical copula process is introduced and the main result of the paper, its weighted weak convergence, is stated. In Section~\ref{sec:app}, the main result is illustratively exploited to derive the asymptotics of multivariate rank statistics and of common estimators for extreme-value copulas. All proofs are deferred to Section~\ref{sec:proofs}, with some auxiliary results postponed to Section~\ref{sec:aux}. Finally, Appendix~\ref{sec:integral} in the supplementary material contains some general results on (locally) bounded variation and integration for two-variate functions which are needed for the proof of Theorem~\ref{theo:score}.

\section{Weighted empirical copula processes} \label{sec:weighted}
\def\theequation{2.\arabic{equation}}
\setcounter{equation}{0}
 
Let $\vect X=(X_1, \dots, X_d)'$ be a $d$-dimensional random vector with joint cumulative distribution function (c.d.f.)~$F$ and continuous marginal c.d.f.s $F_1, \dots ,F_d$. The copula $C$ of $F$, or, equivalently, the copula of $\vect X$, is defined as the c.d.f.\ of the random vector $\vect U = (U_1, \dots, U_d)'$ that arises from marginal application of the probability integral transform, i.e., $U_\dimi=F_\dimi(X_\dimi)$ for  $\dimi=1, \dots, d$. By construction, the marginal c.d.f.s of $C$ are standard uniform on $[0,1]$.
By Sklar's Theorem, $C$ is the unique function for which we have
\[
  F(x_1, \dots, x_d) = C \{ F_1(x_1), \dots, F_d(x_d) \}
\]
for all $\vect x =(x_1, \dots, x_d) \in \R^d$.

Let $\vect X_i, i=1, \dots, n$ be an observed stretch of a strictly stationary time series such that $\vect X_i$ is equal in distribution to $\vect X$. Set $\vect U_i=(U_{i1}, \dots, U_{id})\sim C$  with $U_{i\dimi} = F_\dimi(X_{i\dimi})$.
Define (observable) pseudo observations $\hat{\vect U}_i=(\hat U_{i1}, \dots , \hat U_{id})$ of $C$ through $\hat U_{i \dimi}= n F_{n \dimi}(X_{i\dimi})(/n+1)$ for $i=1, \dots , n$ and $\dimi=1, \dots ,d$. The empirical copula $\hat C_n$ of the sample $\vect X_1, \dots , \vect X_n$ is defined as the empirical distribution function of $\hat{\vect U}_1, \dots , \hat{\vect U}_n$, i.e.,
\[
\hat C_n(\vect u  )= \frac{1}{n}\sum_{i=1}^n \ind (\hat{\vect U}_i \leq \vect u), \qquad \vect u \in [0,1]^d.
\]
The corresponding empirical copula process is defined as 
\[
\vect u \mapsto \hat \Cb_n(\vect u) = \sqrt n \{ \hat C_n (\vect u) - C(\vect u) \}.
\]
For $\omega\ge 0$, define a weight function
\[
g_\omega(\vect u) = 
\min\{ \wedge_{\dimi=1}^d u_\dimi,  \wedge_{\dimi=1}^d [1-(u_1 \wedge \dots \wedge \widehat {u_\dimi} \wedge \dots \wedge u_d)]\}^\omega,
\]
where the hat-notation $u_1 \wedge \dots \wedge \widehat {u_\dimi} \wedge \dots \wedge u_d$ is used as a shorthand for $\min\{u_1, \dots, u_{j-1}, u_{j+1}, \dots, u_d\}$. For $d=2$, the function is particularly nice and reduces to $
g_\omega(u_1, u_2) = \min(u_1, u_2, 1-u_1, 1-u_2)^\omega 
$, see Figure~\ref{fig:level}. Note that for vectors $\vect u \in [0,1]^d$ such that at least one coordinate is equal to $0$ or such that $d-1$ coordinates are equal to~$1$, we have  $g_\omega(\vect u) = 0$. As already mentioned in the introduction for the case $d=2$, these vectors are exactly the points where the limit of the empirical copula process is equal to $0$, almost surely, whence one might hope to obtain a weak convergence result for $\Cb_n/g_\omega$. To prove such a result, a smoothness condition on $C$ has to be imposed.

\begin{condition} \label{cond:2nd}
For every $\dimi \in\{1, \dots, d\}$, the first oder partial derivative $\dot C_\dimi(\vect u) := \partial C(\vect u)/\partial u_\dimi$ exists and is continuous on $V_\dimi=\{ \vect u \in [0,1]^d: u_\dimi \in (0,1) \}$.
For every $\dimi_2, \dimi_2 \in \{1, \dots, d\}$, 
the second order partial derivative $\ddot C_{\dimi_1 \dimi_2}(\vect u) := \partial^2 C(\vect u)/\partial u_{\dimi_1}\partial u_{\dimi_2}$ exists and is continuous on $V_{\dimi_1} \cap V_{\dimi_2}$.  Moreover, there exists a constant $K>0$ such that
\[
	| \ddot C_{\dimi_1\dimi_2}(\vect u) | \le  K \min \left\{ \frac{1}{u_{\dimi_1}(1-u_{\dimi_1})}, \frac{1}{u_{\dimi_2}(1-u_{\dimi_2})} \right\}, \qquad \forall\, \vect u \in V_{\dimi_1} \cap V_{\dimi_2}.
\]
\end{condition}
For completeness, define $\dot C_j(\vect u) = \limsup_{h \to 0} \{ C(\vect u + h \vect e_j) - C(\vect u) \} / h$ wherever it does not exist.
Note, that Condition~\ref{cond:2nd} coincides with Con\-dition~2.1 and Condition 4.1 in \cite{Seg12}, who used it to prove Stute's representation of an almost sure remainder term \citep{Stu84}. The condition is satisfied for many commonly occurring copulas \citep{Seg12}.

For $-\infty \le a < b \le \infty$, let $\mathcal{F}_a^b$ denote the sigma-field generated by those $\vect X_i$ for which $i \in\{a, a+1, \dots, b\}$ and define, for $k\ge 1$,
\begin{align*}
  \alpha^{\scriptscriptstyle [\vect X]}(k)
  = 
  \sup \left\{ \vert \mathbb{P}(A \cap B) - \mathbb{P}(A) \mathbb{P} (B) \vert : 
  {A \in \mathcal{F}_{-\infty}^i,  B \in \mathcal{F}_{i+k}^\infty}, i \in \Z \right\}
\end{align*}
as the alpha-mixing coefficient of the time series $(\vect X_i)_{i\in \Z}$. The sequence is called strongly mixing (or alpha-mixing) if $\alpha^{\scriptscriptstyle [\vect X]}(k) \to 0$ for $k\to\infty$.  Finally, 
\[
\alpha_n(\vect u) = \sqrt n \{ G_n(\vect u) - C(\vect u) \}, \qquad G_n(\vect u) =  n^{-1} {\textstyle \sum_{i=1}^n} \ind( \vect U_i \le \vect u),
\]
denotes the (unobservable) empirical process based on $\vect U_1, \dots, \vect U_n$.

\begin{theorem}[Weighted weak convergence of the empirical copula process] \label{theo:weightalpha}
Suppose that $\vect  X_1, \vect X_2, \dots$ is a stationary, alpha-mixing sequence with $\alpha^{\scriptscriptstyle [\vect X]}(k )=O(a^k)$, as $k \to \infty$, for some $a \in (0,1)$. If the marginals of the stationary distribution are continuous and if the corresponding copula $C$ satisfies Condition~\ref{cond:2nd}, then, for any $\bbb \in (0,1)$ and any $\omega\in(0,1/2)$, 
\[
\sup_{\vect u \in [\frac{\bbb}{n}, 1-\frac{\bbb}{n}]^d} \left| 
\frac{ \hat \Cb_n(\vect u)}{g_\omega(\vect u) } - \frac{\bar \Cb_n(\vect u) } { g_\omega(\vect u) } 
\right| =o_P(1)
\]
where, for any $\vect u \in [0,1]^d$, 
\[
  \bar \Cb_n(\vect u) := \alpha_n(\vect u) - \sum_{\dimi=1}^d  \dot C_\dimi(\vect u) \alpha_n(\vect \udimi ),
\]
with $\vect \udimi = (1, \dots, 1 , u_j, 1, \dots, 1)$. Moreover,  we have
$
\bar \Cb_n/ \tilde g_\omega
\weak 
\Cb_C /  \tilde g_\omega
$
in $(\ell^\infty([0,1]^d), \| \cdot \|_\infty)$, where $\tilde g_\omega(\vect u) = g_\omega(\vect u) + \ind\{g_\omega(\vect u) = 0\}$, where
\[
\Cb_C(\vect u) = \alpha_C(\vect u) - \sum_{\dimi=1}^d  \dot C_\dimi(\vect u) \alpha_C(\vect \udimi ),
\]
and where $\alpha_C$ denotes a tight, centered Gaussian process with covariance
\[
\Cov\{ \alpha_C(\vect u), \alpha_C(\vect v) \} 
= 
\sum\nolimits_{i\in \Z} \Cov \{ \ind( \vect U_0 \le u), \ind( \vect U_i \le v) \}.
\]
\end{theorem} \vspace{-.1cm}
The proof of Theorem~\ref{theo:weightalpha} is given in Section~\ref{subsec:proofmain} below. In fact, we state a more general result which is based on conditions on the usual empirical process~$\alpha_n$. These conditions are subsequently shown to be valid for exponentially alpha-mixing time series.

\section{Applications} \label{sec:app}
\def\theequation{3.\arabic{equation}}
\setcounter{equation}{0}

Theorem~\ref{theo:weightalpha} may be exploited in numerous ways. For instance, many of the most powerful goodness-of-fit tests for copulas are based on distances between the empirical copula and a parametric estimator for $C$ \citep{GenRemBea09}. The results of Theorem~\ref{theo:weightalpha} can be exploited to validate tests for a richer class of distances, as for weighted Kolomogorov-Smirnov or $L^2$-dis\-tances.
Second, estimators for extreme-value copulas can often be expressed through improper integrals involving the empirical copula \citep[see][among others]{GenSeg09}. Weighted weak convergence as in Theorem~\ref{theo:weightalpha} facilitates the anlysis of their asymptotic behavior and allows to extend the available results to time series observations. Details regarding the CFG- and the Pickands estimator are worked out in Section~\ref{subsec:pick} below.

Theorem~\ref{theo:weightalpha} may also be used outside the genuine copula framework, for instance, for proving asymptotic normality of multivariate rank statistics. The power of that approach lies in the fact that proofs for time series are essentially the same as for i.i.d.\ data sets. In Section~\ref{subsec:rankstat}, we derive a general weak convergence result for bivariate rank statistics.

\subsection{Bivariate rank statistics} \label{subsec:rankstat}
Bivariate rank statistics constitute an important class of real-valued statistics that can be written as
\begin{align*}
  R_n= \frac{1}{n}\sum_{i=1}^n J(\hat U_{i1}, \hat U_{i2})
\end{align*}
for some function $J:(0,1)^2 \to \mathbb{R}$, called score function. $R_n$ can also be expressed as a Lebesgue-Stieltjes integral with respect to $\hat C_n$, i.e.,
\[
R_n = \int_{[\frac{1}{n+1},\frac{n}{n+1}]^2} J(u,v) \mathrm{d}\hat C_n(u,v),
\]
which offers the way to derive the asymptotic behavior of $R_n$ from the asymptotic behavior of the empirical copula. This idea has already been exploited in \cite{FerRadWeg04}: however, in their Theorem 6, $J$ has to be a bounded function which is not the case for many interesting examples. Also, the uniform central limit theorems for multivariate rank statistics in \cite{VanWel07} require rather strong smoothness assumptions on $J$ (which imply boundedness of $J$).

\begin{example}[Rank Autocorrelation Coefficients] \label{ex:rankauto}
Suppose $Y_1, \dots, Y_n$ are drawn from a stationary, univariate time series  $(Y_i)_{i\in\Z}$. Rank autocorrelation coefficients of lag $k\in\N$ are statistics of the form 
\[
r_{n,k}  =\frac{1}{n-k}\sum_{i=k+1}^n J_1\Big \{\tfrac{n}{n+1}F_{n}(Y_{i}) \Big \}J_2\Big \{\tfrac{n}{n+1}F_{n}(Y_{i-k})\Big\}, 
\]
where $J_1, J_2$ are real-valued functions on $(0,1)$ and $F_n$ denotes the empirical cdf of $Y_1, \dots, Y_n$. For example, the van der Waerden autocorrelation \citep{HalPur88} is given by 
\begin{align*}
r_{n,k,vdW}=\frac{1}{n-k}\sum_{i=k+1}^n \Phi^{-1}\Big \{\tfrac{n}{n+1}F_{n}(Y_{i})\Big \}\Phi^{-1}\Big \{\tfrac{n}{n+1}F_{n}(Y_{i-k})\Big \},
\end{align*}
(with $\Phi$ and $\Phi^{-1}$ denoting the cdf of the standard normal distribution and its inverse, respectively)
and the Wil\-coxon autocorrelation \citep{HalPur88} is defined as
\begin{align*}
r_{n,k,W}= \frac{1}{n-k}\sum_{i=k+1}^n \Big  \{\tfrac{n}{n+1}F_{n}(Y_{i})-\frac{1}{2}\Big \}\log\Big \{\frac{\tfrac{n}{n+1}F_{n}(Y_{i-k}) }{1-\tfrac{n}{n+1}F_{n}(Y_{i-k})} \Big \}.
\end{align*}
Obviously, the corresponding score functions are unbounded.
Asymptotic normality for these and similar rank statistics has been shown for i.i.d.~observations and for $ARMA$-processes \citep{HalIngPur85}. To the best of our knowledge, no general tool to handle the asymptotic behavior of such statistics for dependent observations seems to be available. Theorem \ref{theo:score} below aims at partially filling that gap.
\end{example}

\begin{example}[The pseudo-maximum likelihood estimator]
As a common practice in bivariate copula modeling one assumes to observe a sample $\vect X_1, \dots, \vect X_n$ from a bivariate distribution whose copula belongs to a parametric copula family, parametrized by a finite-dimensional parameter $\theta\in \Theta \subset  \R^p$. Except for the assumption of absolute continuity, the marginal distributions are often left unspecified in order to allow for maximal robustness with respect to potential miss-specification. In such a setting, the pseudo-maximum likelihood estimator (see \cite{GenGhoRiv95} for a theoretical investigation) provides  the most common estimator for the parameter~$\theta$. If $c_\theta$ denotes the corresponding copula density, the estimator is defined as
\[
\hat \theta_n={\arg\max}_{\theta \in \Theta} \sum_{i=1}^{n} \log  \{ c_\theta (\hat U_{i1}, \hat U_{i2} ) \}.
\]
Using standard arguments from maximum-likelihood theory and imposing suitable regularity conditions, the asymptotic distribution of $\sqrt n (\hat \theta_n - \theta_0)$ can be derived from the asymptotic behavior of
\begin{align}  \label{eq:mlrn}
R_n = \frac{1}{n} \sum_{i=1}^{n} J_{\theta_0} (\hat U_{i1}, \hat U_{i2} ),
\end{align}
where $\theta_0$ denotes the unknown true parameter and where $J_\theta= (\partial \log c_\theta )/(\partial \theta)$ denote the score function. Typically, this function is unbounded, as for instance in case of the bivariate Gaussian copula model where $\theta$ is the correlation coefficient and the score function takes the form
\[
J_{\theta } (u,v)= \frac{\theta(1-\theta^2)-\theta\{ \Phi^{-1}(u)^2 + \Phi^{-1}(v)^2\} + (1+\theta^2)\Phi^{-1}(u)\Phi^{-1}(v)}{1+\theta^2}.
\]
Still, the conditions of Theorem~\ref{theo:score} below can be shown to be valid.

Finally, note that pseudo-maximum likelihood estimators also arise in Markovian copula models \citep{CheFan06mar} where copulas are used to model the serial dependence of a stationary time series at lag one. 
Again, their asymptotic distribution may be derived from rank statistics as in \eqref{eq:mlrn}.
\end{example}

The following theorem is the central result of this section. It establishes weak convergence of bivariate rank-statistics by exploiting weighted weak convergence of the empirical copula process. For the definition of the space of functions of locally bounded total variation in the sense of Hardy-Krause, $BVHK_{loc}((0,1)^2)$, and for Lebesgue-Stieltjes integrals with respect to such functions, we refer the reader to Definition~\ref{def:locbvhk} in the supplementary material. The proof is given in Section~\ref{subsec:proofs2}.

\begin{theorem}\label{theo:score}
Suppose the conditions of Theorem~\ref{theo:weightalpha} are met. Moreover, suppose that $J \in BVHK_{loc}((0,1)^2)$ is right-continuous and that there exists $\omega\in(0,1/2)$ such that $\vert J (\vect u) \vert \leq \const \times g_\omega (\vect u)^{-1}$  and such that 
\begin{align}
\int_{(0,1)^2} g_\omega(\vect u) \vert \mathrm{d} J  (\vect u)\vert< \infty. \label{cond:score}
\end{align}
Moreover, for $\delta \to 0$, suppose that
\begin{align}
\int_{(\delta,1-\delta]} \vert  J( \mathrm{d}u,\delta)  \vert  =O(\delta^{-\omega}) & \text{ and } \int_{(\delta,1-\delta]}   \vert  J(\mathrm{d}u,1-\delta)  \vert   =O(\delta^{-\omega}), \label{cond:score2}\\
\int_{(\delta,1-\delta]}  \vert  J (\delta,\mathrm{d}v) \vert   =O(\delta^{-\omega}) & \text{ and } \int_{(\delta,1-\delta]}  \vert J(1-\delta,\mathrm{d}v)  \vert   =O(\delta^{-\omega}).\label{cond:score3}
\end{align}
Then, as $n \to \infty$,
\begin{align*}
\sqrt n \{ R_n - \mathbb{E} [ J(\vect U) ]  \} \weak   \int_{(0,1)^2  }\mathbb{C}_C (\vect u) \mathrm{d}  J(\vect u)  .
\end{align*}
The weak limit is normally distributed with mean $0$ and variance
\[
\sigma^2 =  \int_{(0,1)^2  } \int_{(0,1)^2  } \Exp[\Cb_C(\vect u) \Cb_C(\vect v)] \mathrm{d} J(\vect u) \mathrm{d}J(\vect v).
\]
 \end{theorem}

\begin{remark} \label{rem:rank}
(i)
Provided the second order partial derivative $\ddot J_{12} (u,v) := \partial^2 J(u,v)/\partial u\partial v$ exists, then the conditions \eqref{cond:score}--\eqref{cond:score3} are equivalent to 
$
\int_{(0,1)^2} g_\omega(u,v) \vert \ddot J_{12} (u,v) \vert \mathrm{d}(u,v)< \infty 
$
and, as $\delta \to 0$,
\begin{align*}
\int_\delta^{1-\delta} \vert \dot J_1(u,\delta)  \vert  \mathrm{d}u =O(\delta^{-\omega}) & \text{ and } \int_\delta^{1-\delta} \vert \dot J_1(u,1-\delta)  \vert  \mathrm{d}u =O(\delta^{-\omega}),\\
\int_\delta^{1-\delta} \vert \dot J_2 (\delta,v) \vert  \mathrm{d}v =O(\delta^{-\omega}) & \text{ and } \int_\delta^{1-\delta} \vert \dot J_2(1-\delta,v)  \vert  \mathrm{d}v =O(\delta^{-\omega}),
\end{align*} 
where $\dot J_1(u,v) := \partial J(u,v)/\partial u, \dot J_2(u,v) := \partial J(u,v)/\partial v$.

\smallskip
\noindent
(ii) A careful check of the proof of Theorem~\ref{theo:score} shows that the theorem actually remains valid under the more general conditions of Theorem~\ref{theo:weight} below, with $\omega \in (0,1/2)$ replaced by $\omega \in(0,\frac{\theta_1}{2(1-\theta_1)}\wedge \frac{\theta_2}{2(1-\theta_2)}\wedge (\theta_3-1/2))$.
\end{remark}

As a simple application of Theorem~\ref{theo:score} let us return to the autocorrelation coefficients from Example~\ref{ex:rankauto}. It can easily be shown that both $J_{vdW}(u,v)=\Phi^{-1}(u)\Phi^{-1}(v)$ 
and $J_{W}(u,v)=(u-\frac{1}{2})\log(\frac{v}{1-v})$ satisfy the conditions of Theorem \ref{theo:score}. To prove this for $J_{vdW}$ use that $\vert \Phi^{-1}(u)\vert \leq \{ u (1-u)\}^{-\eps}$  for any $\eps>0$ and that $\frac{1}{ \phi \{ \Phi^{-1}(u)\} }\leq \{ u (1- u)\}^{-1}$, with $\phi$ denoting the density of the standard normal distribution. Therefore, both coefficients are asymptotically normally distributed for any stationary, exponentially alpha-mixing time series provided that the copula of $(Y_t,Y_{t-k})$ satisfies Condition~\ref{cond:2nd}. This broadens results from \cite{HalIngPur85}, which may be further extended along the lines of Remark~\ref{rem:rank}(ii) by a more thorough investigation of Conditions~\ref{cond:modulus}--\ref{cond:quantile}. Details are omitted for the sake of brevity.

\subsection{Nonparametric estimation of Pickands dependence function}\label{subsec:pick}
Theorem~\ref{theo:weightalpha} can be used to extend recent results for the estimation of Pickands dependence functions. Recall that $C$ is a multivariate extreme-value copula if and only if $C$ has a representation of the form
\[ 
C(\vect u) = \exp \left\{ \Big( \sum_{\dimi=1}^d \log u_\dimi\Big) A  \Big ( \frac{\log u_1}{\sum_{\dimi=1}^d \log u_\dimi}, \dots, \frac{\log u_{d-1}}{\sum_{\dimi=1}^d \log u_\dimi} \Big  ) \right\}, ~~ \vect u \in (0,1)^d,
\] 
for some function $A:\Delta_{d-1} \to [1/d,1]$, where $\Delta_{d-1}$ denotes the unit simplex $\Delta_{d-1}= \{ \vect w=(w_1, \dots , w_{d-1}) \in [0,1]^{d-1}: \sum_{\dimi=1}^{d-1}w_\dimi\leq 1\}$. In that case, $A$ is necessarily convex and satisfies the relationship
\[
\max(w_1, \dots, w_d) \le A(w_1, \dots, w_{d-1}) \le 1 \quad (w_d = {\textstyle 1-\sum_{j=1}^{d-1} w_j}),
\]
for all $\vect w \in \Delta_{d-1}$. By reference to \cite{Pic81}, 
$A$ is called Pickands dependence function. Nonparametric estimation methods for $A$ in the i.i.d.\ case and under the additional assumption that the marginal distributions are known have been considered in \cite{Pic81, Deh91, CapFouGen97, JimVilFlo01}, among others. In the more realistic case of unknown marginal distribution, rank-based estimators have for instance been investigated in \cite{GenSeg09, BucDetVol11, GudSeg12, BerBucDet13}, among others. 
For illustrative purposes, we restrict attention to the rank-based versions of the Pickands estimator in \cite{GudSeg12} in the following, even though the results easily carry over to, for instance, the CFG-estimator. The Pickands-estimator is defined as
\[
\hat A^P_n (\vect w) = \left[ \frac{1}{n} \sum_{i=1}^n \min\Big\{ \frac{- \log(\hat U_{i1}) }{ w_1} , \dots, \frac{- \log (\hat U_{id})}{w_{d} } \Big \} \right]^{-1} 
\]
and it follows by simple algebra (see Lemma 1 in \citealp{GudSeg12}) that
$
\Ab_n^P
:=\sqrt n ( \hat A_{n}^P - A ) 
=
- A^2 \Bb_n^P/ (1+n^{1/2} \Bb_n^P),
$
where
\[
\Bb_n^P(\vect w) = \int_0^1 \hat \Cb_n(u^{w_1}, \dots, u^{w_{d}}) \, \frac{\mathrm d u}{u}.
\]
Note that $\int_0^1u^{-1} \, \mathrm d u$ does not converge, which hinders a direct application of the continuous mapping theorem to deduce weak convergence of $\Bb_n^P$ (and hence of $\Ab_n^P$) in $\ell^\infty(\Delta_{d-1})$ just on the basis of (unweighted) weak convergence of $\hat \Cb_n$. Deeper results are necessary and in fact, 
\cite{GenSeg09} and \cite{GudSeg12} deduce weak convergence  of $\Bb_n^P$ by using Stute's representation for the empirical copula process based on i.i.d.\ observations (see \citealp{Stu84, Tsu05}) and by exploiting a weighted weak convergence result for $\alpha_n$. 

With Theorem~\ref{theo:weightalpha}, we can give a much simpler proof. Write
\[
\Bb_n^P(\vect w) = \int_{0}^1 \frac{\hat \Cb_n( u^{w_1}, \dots, u^{w_{d}})}{\min( u^{w_1}, \dots, u^{w_{d}})^\omega}  \frac{\min(u^{w_1}, \dots, u^{w_{d}})^\omega}{u}\, \mathrm d u.
\]
Then, since $\int_0^1 {\min( u^{w_1}, \dots, u^{w_{d}})^\omega}\, \frac{\mathrm d u}{u} \le \int_0^1 u^{\omega/d-1}\, \mathrm d u$ exists for any $\omega>0$, weak convergence of $\Bb_n^P$ is a direct consequence of the continuous mapping theorem and Theorem~\ref{theo:weightalpha}. Note that this method of proof is not restricted to the i.i.d.\ case.

\section{Proofs}\label{sec:proofs}
\def\theequation{4.\arabic{equation}}
\setcounter{equation}{0}

\subsection{Proof of Theorem~\ref{theo:weightalpha}} 
\label{subsec:proofmain}


Theorem~\ref{theo:weightalpha} will be proved by an application of a more general result on the empirical copula process. For its formulation, we need a couple of additional conditions which, subsequently, will be shown to be satisfied for exponentially alpha-mixing time series.

\begin{condition} \label{cond:modulus}
There exists some $\theta_1 \in(0,1/2]$ such that, for all $\mu \in (0,\theta_1)$ and all sequences $\delta_n \to 0$, we have
\[
	M_n(\delta_n, \mu)
	:=
	\sup_{ | \vect u - \vect v | \le \delta_n}   \frac{\vert  \alpha_n(\vect u) - \alpha_n(\vect v) \vert }{\vert \vect u - \vect v \vert^\mu\vee n^{-\mu}}  =o_P(1).
\]
\end{condition}

Condition~\ref{cond:modulus} can for instance be verified in the i.i.d.\ case with $\theta_1=1/2$, exploiting a bound for the multivariate oscillation modulus derived in Proposition A.1 in \cite{Seg12}.

\begin{condition} \label{cond:weak}
The empirical process $\alpha_n$ converges weakly in $\ell^\infty([0,1]^d)$ to some limit process $\alpha_C$ which has continuous sample paths, almost surely.
\end{condition}

For i.i.d.\ samples, the latter condition is satisfies with $\alpha_C$ being a $C$-Brownian bridge, i.e., a centered Gaussian process with con\-tinuous sample paths, a.s., and with $\Cov\{ \alpha_C(\vect u), \alpha_C(\vect v) \} = C(\vect u \wedge \vect v) - C(\vect u)C(\vect v)$.

\begin{condition} \label{cond:quantile}
There exist  $\theta_2 \in(0,1/2]$ and  $\theta_3 \in (1/2, 1]$  such that, for any $\omega \in (0,\theta_2)$, any $\lambda\in (0, \theta_3)$ and all $j=1, \dots, d$, we have 
\[
	\sup_{u_j\in(0,1)} \left| \frac{\alpha_{nj}(u_j) }{u_j^{\omega}(1-u_j)^{\omega}} \right| =O_P(1), \quad 
	\sup_{u_j\in(1/n^\lambda,1-1/n^\lambda)} \left| \frac{\beta_{nj}(u_j) }{u_j^{\omega}(1-u_j)^{\omega}} \right| =O_P(1),
\]
where
$\alpha_{nj}(u_j) = \sqrt n \{ G_{nj}(u_j) - u_j \}
$ and $\beta_{nj}(u_j) = \sqrt n \{ G_{nj}^-(u_j) - u_j \}$.
\end{condition}
Here, $
G_{n\dimi}(u_\dimi) = n^{-1} \sum_{i=1}^n \ind( U_{i\dimi} \le u_\dimi)
$
and,
for a distribution function $H$ on the reals, $H^-$ denotes the (left-continuous) generalized inverse function of $H$
defined as 
\begin{align*} 
  H^{-}(u):= \inf\{x\in\R : H(x)\geq u \}, \quad  0<u\leq 1,
\end{align*}
and $H^-(0) = \sup\{x\in\R : H(x)=0\}$.
In the i.i.d.~case, Condition~\ref{cond:quantile} is a mere consequence of results in \cite{CsoCsoHorMas86}, with $\theta_2=1/2$ , $\theta_3=1$.

The following proposition shows that the (probabilistic) Conditions~\ref{cond:modulus}, \ref{cond:weak} and \ref{cond:quantile} are satisfied for sequences that are exponentially alpha-mixing. 

\begin{prop}\label{prop:wdcheckmod}~
Suppose that $\vect  X_1, \vect X_2, \dots$ is a stationary, alpha-mixing sequence with $\alpha^{\scriptscriptstyle [\vect X]}(k )=O(a^k)$, as $k \to \infty$, for some $a \in (0,1)$. Then, Conditions \ref{cond:modulus}, \ref{cond:weak} and \ref{cond:quantile} are satisfied with $\theta_1=\theta_2=1/2$ and $\theta_3=1$.
\end{prop} 

Here, Condition~\ref{cond:quantile} is a mere consequence of results in  \cite{ShaYu96} and \cite{CsoYu96}, whereas Condition~\ref{cond:weak} has been shown in \cite{Rio00}. For the proof of Condition~\ref{cond:modulus}, we can rely on results from \cite{KleVolDetHal14}. 
The precise arguments are given in Section~\ref{subsec:alphacond} below.

The following theorem can be regarded as a generalization of Theorem~\ref{theo:weightalpha}: weighted weak convergence of the empirical copula process takes place provided the abstract Conditions~\ref{cond:modulus}, \ref{cond:weak} and \ref{cond:quantile} are met. The proof is given in Section~\ref{subsec:proofweight} below.

\begin{theorem}[Weighted weak convergence of empirical copula processes] \label{theo:weight}
Suppose Conditions~\ref{cond:2nd}, \ref{cond:modulus} and \ref{cond:quantile} are met. Then, for any $\bbb \in (0,1)$ and any $\omega\in(0,\frac{\theta_1}{2(1-\theta_1)}\wedge \frac{\theta_2}{2(1-\theta_2)}\wedge (\theta_3-1/2))$, 
\[
\sup_{\vect u \in [\frac{\bbb}{n}, 1-\frac{\bbb}{n}]^d} \left| 
\frac{ \hat \Cb_n(\vect u)}{g_\omega(\vect u) } - \frac{\bar \Cb_n(\vect u) } { g_\omega(\vect u) } 
\right| =o_P(1).
\]
If additionally Condition~\ref{cond:weak} is met, then
$
\bar \Cb_n/ \tilde g_\omega
\weak 
\Cb_C /  \tilde g_\omega
$
in $(\ell^\infty([0,1]^d), \| \cdot \|_\infty)$.
\end{theorem}

\begin{proof}[Proof of Theorem~\ref{theo:weightalpha}.]
The theorem is a mere consequence of Proposition~\ref{prop:wdcheckmod} and Theorem~\ref{theo:weight}.
\end{proof}

\subsection{Proof of Proposition~\ref{prop:wdcheckmod}} \label{subsec:alphacond}
For an $r$-dimensional random vector $(Y_1, \dots Y_r)'$, define the $r$th order joint cumulant by
\[
\cum(Y_1, \dots Y_r) = \sum_{\{ \nu_1, \dots , \nu_p\}} (-1)^{p-1} (p-1)! \mathbb{E}\big  (\prod_{j \in \nu_1} Y_j\big) \times  \dots \times \mathbb{E}\big(\prod_{j \in \nu_p} Y_j\big),
\]
where the summation extends over all partitions $\{ \nu_1, \dots , \nu_p \}$, $p\in \{1, \dots , r\}$, of $\{ 1, \dots,r \}$. The following lemma will be one of the main tools for establishing Condition~\ref{cond:modulus} under exponentially alpha-mixing.

\begin{lemma} \label{lem:kleyvoly}
If $Y_1, Y_2, \dots$ is a strictly stationary sequence of random variables with $\vert Y_i \vert \leq K < \infty$ and if there exist constants $\rho \in (0,1)$ and $K' < \infty$ such that for any $p \in \mathbb{N}$ and arbitrary $i_1, \dots , i_p \in \mathbb{Z}$
\[
\vert \cum (Y_{i_1}, \dots , Y_{i_p})\vert \leq K' \rho^{\max_{k,\ell} \vert i_k - i_\ell\vert},
\]
then, there exist constants $C_1, C_2< \infty$ only depending on $K,K'$ and $\vert \nu_r \vert$ such that
\[
\Big \vert \cum \Big (\sum_{i=1}^n Y_i, j \in \nu_r\Big ) \Big \vert \leq C_1(n+1)\eps (\vert \log \eps \vert +1 )^{C_2},
\]
where $\eps = \mathbb{E}[ \vert Y_i \vert ]$.
\end{lemma}

\begin{proof}
The proof is almost identical to the proof of Lemma 7.4 in \cite{KleVolDetHal14} and is therefore omitted.
\end{proof}

\begin{proof}[Proof of Proposition~\ref{prop:wdcheckmod}]
The weak convergence result in Condition~\ref{cond:weak} has been shown in Theorem 7.3 in \cite{Rio00}.

Regarding Condition~\ref{cond:quantile}, note that exponentially alpha-mixing implies that $\alpha^{\scriptscriptstyle [\vect X]}(k)=O(k^{-b-\delta})$ for any $b>1+\sqrt 2$ and any $\delta>0$. Therefore, by Theorem~3.1 in \cite{ShaYu96}, 
\[
	\sup_{u\in[0,1]} \left| \frac{\sqrt n \{ G_{nj}(u) - u \} }{ \{u(1-u)\}^{(1-1/b)/2}} \right| =O_P(1) 
\]
Since $(1-1/b)/2$ converges to $1/2$ for $b\to \infty$, we indeed have the first display in Condition~\ref{cond:quantile} with $\theta_2=1/2$. Regarding the second display,  \cite{CsoYu96} have shown that
\begin{align*}
\sup_{u \in [\delta_n , 1-\delta_n]} \left| \frac{ \sqrt n \{ G_{n \dimi }^-(u) -u \}}{\{u(1-u)\}^{(1-1/b)/2} } \right| =O_P(1),
\end{align*}
for $\delta_n=n^{-b/(1+b)} \to n^{-1}$ as $b \to \infty$, which implies that we may choose $\theta_3=1$.

Finally, consider Condition~\ref{cond:modulus}. 
It follows from a simple multivariate extension of Proposition~3.1 in \cite{KleVolDetHal14} that, in our case of an exponentially alpha-mixing sequence $(\vect X_i)_{i \in \mathbb{Z}}$,
there exist constants $\rho \in (0,1)$ and $K < \infty$ such that, for any $p\in\N$ and any arbitrary hyper-rectangles $A_1, \dots , A_p \subset \mathbb{R}^d$ and arbitrary $i_1, \dots ,i_p \in \mathbb{Z}$,
\begin{align} \label{eq:cum}
\vert \cum (\ind \{ \vect X_{i_1}\in A_1 \}, \dots , \ind \{ \vect  X_{i_p} \in A_p\}) \vert \leq K \rho^{\max_{k,\ell} \vert  i_k - i_\ell \vert}.
\end{align}
The latter display will be the main tool to establish  Condition~\ref{cond:modulus}.
First, decompose
\begin{align*} 
M_n(\delta_n, \mu)
= 
\sup_{\vert \vect u - \vect v \vert \leq \delta_n} \frac{\vert \alpha_n(\vect u ) - \alpha_n (\vect v) \vert }{\vert \vect u - \vect v \vert^\mu \vee n^{-\mu}} 
= 
\max\{S_{n1}, S_{n2}\}
\end{align*}
where
\begin{align*}
S_{n1}= \sup_{n^{-1} \leq \vert\vect u - \vect v \vert \leq \delta_n} \frac{\vert \alpha_n(\vect u) - \alpha_n(\vect v) \vert }{\vert \vect u -\vect v \vert^\mu}, 
\quad
S_{n2}  =\sup_{ \vert \vect u - \vect v \vert \leq n^{-1}}  n^\mu \vert \alpha_n(\vect u) - \alpha_n (\vect u) \vert .
\end{align*}
It suffices to show that $S_{n1}=o_P(1)$ and $S_{n2}=o_P(1)$ as $n \to \infty$.

First consider $S_{n2}$. 
We will show that, for any $\ell \in \mathbb{N}$ and any $\beta \in (0,1)$, there exist constants $K_1$ and $K_2$ only depending on $d$, $\ell$, $\beta$ and the constants in \eqref{eq:cum} such that 
\begin{align}\label{smallvalues}
\Prob\Big(\sup_{\vert \vect u - \vect v \vert \leq n^{-1}} \vert \alpha_n(\vect u)- \alpha_n(\vect v) \vert > \eps\Big) 
\le 
3 \ind(n^{-1/2}>K_1 \eps) + K_2 \eps^{-2\ell} n^{1-\beta \ell}.
\end{align}
Indeed, $S_{n2}=o_P(1)$ follows by setting $\eps = n^{-\mu} \eps'$, by choosing $\beta>2\mu$ and by finally choosing $\ell$ sufficiently large.

In order to prove \eqref{smallvalues}, we begin by bounding the left-hand side of that display by
\begin{multline*}
 \Prob \Big(\sup_{\vert  \vect u - \vect v \vert \leq n^{-1}} \Big \vert \frac{1}{\sqrt n } \sum_{i=1}^n \ind (\vect U_i \leq \vect u) - \ind(\vect U_i \leq  \vect v) \Big\vert >\frac{\eps}{2}\Big) \\
+ \Prob \Big(\sup_{\vert \vect u - \vect v \vert \leq n^{-1}} \sqrt n \vert C(\vect u ) - C(\vect v) \vert > \frac{\eps }{2} \Big),
\end{multline*}
where the second probability is smaller than $ \ind(n^{-1/2}> \frac{\eps}{2}) $ by Lipschitz-continuity of $C$. 
Furthermore, we have
\begin{align*}
&\hspace{-1cm} \sup_{\vert  \vect u - \vect v \vert < n^{-1}} \Big \vert \frac{1}{\sqrt n } \sum_{i=1}^n \ind(\vect U_i \leq \vect u) - \ind(\vect U_i \leq \vect v) \Big \vert \\
&\leq  \sum_{\dimi=1}^d \sup_{0 \leq v_\dimi - u_\dimi \le n^{-1}} \frac{1}{\sqrt n } \sum_{i=1}^n \ind (U_{i\dimi} \le v_\dimi) - \ind(U_{i\dimi} \le u_\dimi)  \\
&= \sum_{\dimi=1}^d \sup_{0 \leq v_\dimi - u_\dimi \le n^{-1}} \sqrt n \{ G_{n\dimi}(v_\dimi) -  G_{n\dimi}(u_\dimi)\} \\
&\le \sum_{\dimi=1}^d \sup_{|v_\dimi - u_\dimi |\le n^{-1}} \sqrt n | G_{n\dimi}(v_\dimi) -  G_{n\dimi}(u_\dimi) - (v_\dimi-u_\dimi) | + \frac{d}{\sqrt n} \\
&= \sum_{\dimi=1}^d \sup_{|v_\dimi - u_\dimi |\le n^{-1}} |\alpha_{n\dimi}(v_\dimi)-\alpha_{n\dimi}(u_\dimi) | + \frac{d}{\sqrt n} .
\end{align*}
We now proceed similar as in the proof of Lemma 8.6 in \cite{KleVolDetHal14} to bound the sum on the right-hand side. Set $M_n= \{ 0,\frac{1}{n}, \frac{2}{n},\dots,1 \} $. Monotonicity of $G_{n \dimi}$ yields 
\begin{multline*}
 \sup_{\vert u_j -v_j \vert \leq n^{-1}} \sqrt n \vert G_{n \dimi }(u_j) - G_{n\dimi }(v_j) -(u_j-v_j) \vert   \\
\le  
\max_{u_j, v_j \in M_n : \vert  u_j - v_j \vert \leq 2/n} \sqrt n \vert G_{n \dimi }(u_j) - G_{n \dimi }(v_j) - (u_j-v_j) \vert  + 2/\sqrt n.
\end{multline*}
Therefore, we get
\begin{multline*}
  \Prob \Big(\sup_{\vert  \vect u - \vect v \vert \leq n^{-1}} \Big \vert \frac{1}{\sqrt n } \sum_{i=1}^n \ind (\vect U_i \leq \vect u) - \ind(\vect U_i \leq  \vect v) \Big \vert >\frac{\eps}{2}\Big) \\
\leq \Prob\Big(\sum_{\dimi =1}^d \max_{u_j, v_j \in M_n : \vert  u_j - v_j \vert \leq 2/n} |\alpha_{n\dimi}(v_\dimi)-\alpha_{n\dimi}(u_\dimi) |  > \frac{\eps}{6} \Big)  \\
+\ind \Big( n^{-1/2}> \frac{\eps}{6d}\Big)+ \ind \Big(n^{-1/2}> \frac{\eps }{ 12}\Big) .
\end{multline*}
Now, note that the set $\{ (u,v)\in M_n^2 : \vert u-v \vert \leq 2/n \}$ contains $O(n)$ elements. Since  $\mathbb{E}(\max_{i=1, \dots , m} \vert Y_i \vert^p)\leq m \times \max_{i=1, \dots, m} \mathbb{E}(\vert Y_i \vert^p)$  for any random variables $Y_1, \dots , Y_m$, we can conclude that  
\begin{align*}
&\hspace{-0.8cm} \Prob\Big(\sum_{\dimi =1}^d \max_{u_j, v_j \in M_n : \vert  u_j - v_j \vert \leq 2/n} |\alpha_{n\dimi}(v_\dimi)-\alpha_{n\dimi}(u_\dimi) |> \frac{\eps }{8}\Big)\\
&\le \sum_{\dimi =1}^d \Prob\Big( \max_{u_j, v_j \in M_n : \vert  u_j - v_j \vert \leq 2/n} |\alpha_{n\dimi}(v_\dimi)-\alpha_{n\dimi}(u_\dimi) |> \frac{\eps }{8d}\Big)\\
&\le \sum_{\dimi =1}^d (8d)^{2\ell} \eps^{-2\ell} \mathbb{E}\Big [  \max_{u_j, v_j \in M_n : \vert  u_j - v_j \vert \leq 2/n} |\alpha_{n\dimi}(v_\dimi)-\alpha_{n\dimi}(u_\dimi) |^{2\ell}\Big] \\
&\le \const \times \eps^{-2\ell} \sum_{\dimi =1}^d  n  \sup_{ \vert  u_j - v_j \vert \leq 2/n} \mathbb{E}\Big [ |\alpha_{n\dimi}(v_\dimi)-\alpha_{n\dimi}(u_\dimi) |^{2\ell}\Big].
\end{align*}
The assertion in \eqref{smallvalues} now follows from an inequality in the proof of Lem\-ma~8.6 in \cite{KleVolDetHal14}. These authors showed that, if \eqref{eq:cum} is satisfied, then, for any $\ell \in \mathbb{N}$, there exist constants $c_1$ and $c_2$ which only depend on $\ell$ and the constants in \eqref{eq:cum} such that
\begin{align*}
\sup_{u_j,v_j \in [0,1]: \vert u_j-v_j \vert \leq \delta} \mathbb{E} \vert \alpha_{nj}(u_j) - \alpha_{nj}(v_j) \vert^{2\ell}  \leq c_2  [ \{ \delta ( 1+ \vert  \log \delta  \vert^{c_1}) \} \vee n^{-1 } ]^{\ell }.
\end{align*}
Set $\delta=2/n$ and exploit that $\log n \le n^{(1-\beta)/c_1}$ for $\beta \in(0,1)$ to get rid of the logarithmic term on the right-hand side to finally  arrive at \eqref{smallvalues}.

It remains to be shown that $S_{n1}=o_P(1)$. We have
\begin{align*} 
&\hspace{-0.8cm}\Prob \Big (\sup_{n^{-1}\leq \vert \vect u - \vect v \vert \leq \delta_n} \frac{\vert  \alpha_n(\vect u) - \alpha_n(\vect v) \vert }{\vert \vect u - \vect v \vert ^\mu} > \eps \Big ) \nonumber \\
\le& \ \Prob \Big ( \max_{k:n^{-1} <2^{-k}\delta_n } \ \sup_{2^{-(k+1)}\delta_n \leq \vert \vect u - \vect v \vert \leq 2^{-k}\delta_n} \frac{\vert \alpha_n(\vect u) -\alpha_n(\vect v) \vert}{ \vert \vect u - \vect v \vert^\mu}>\eps \Big )\nonumber\\
\leq &\sum_{k:n^{-1}< 2^{-k}\delta_n } \Prob \Big ( \sup_{2^{-(k+1)}\delta_n \leq \vert \vect u- \vect v \vert \leq 2^{-k}\delta_n} \vert \alpha_n(\vect u) - \alpha_n(\vect v) \vert > \eps (2^{-k}\delta_n)^\mu 2^{-\mu} \Big)\nonumber\\
\leq & \sum_{k: n^{-1}< 2^{-k} \delta_n} \Prob \Big ( \sup_{\vert \vect u - \vect v \vert \leq 2^{-k}\delta_n} \vert \alpha_n(\vect u) - \alpha_n(\vect v) \vert > \eps (2^{-k}\delta_n)^\mu 2^{-\mu} \Big). 
\end{align*}
Therefore, we only have to show, that
\begin{align} \label{eq:sumksum}
\sum_{k: n^{-1}< 2^{-k} \delta_n} \Prob \Big ( \sup_{\vert \vect u - \vect v \vert \leq 2^{-k}\delta_n} \vert \alpha_n(\vect u) - \alpha_n(\vect v) \vert > \eps (2^{-k}\delta_n)^\mu 2^{-\mu} \Big) =o(1).
\end{align}

We will show later that, for any $L \in \mathbb{N}$ and for any $\gamma \in (0,1/2)$,  there exists a constant $K=K(\gamma, L)>0$ such that, for all $\vect u , \vect v\in [0,1]^d$ with $\vert \vect u - \vect v \vert  \geq n ^{-1}$, 
\begin{align} \label{eq:condmoment}
\Vert \alpha_n(\vect u ) - \alpha_n(\vect v) \Vert_{2L} \leq K \vert \vect u - \vect v \vert^\gamma =: K \mathrm{d}(\vect u ,\vect v),
\end{align}
where $\Vert X \Vert_p=\mathbb{E} [\vert X\vert ^p]^{1/p}$.
Note that the packing number $D(\eps, \mathrm{d})$ of the metric space $([0,1]^d, \mathrm{d})$ satisfies $D(\eps, \mathrm{d}) \le \const \times \eps^{-d / \gamma}$. Then,
using the notation $\Psi (x) = x^{2L}$, $\Psi^{-1}(x)=x^{1/(2L)}$, $\delta= (2^{-k}\delta_n)^\gamma$ and $\bar \eta=2n^{-\gamma}$, Lemma~7.1 in \cite{KleVolDetHal14} yields the existence of a random variable $S_1$  such that 
\begin{multline*}
 \Prob \Big ( \sup_{\vert \vect u - \vect v \vert \leq 2^{-k}\delta_n} \vert \alpha_n(\vect u) - \alpha_n(\vect v) \vert > \eps (2^{-k}\delta_n)^\mu 2^{-\mu} \Big) \\
\leq 
\Prob\Big(S_1 > (2^{-k}\delta_n)^\mu2^{-\mu -1} \eps\Big) \hspace{5cm}\\
+ \Prob\Big(2 \sup_{\vert \vect u - \vect v \vert \leq n^{-1} , \vect u \in \tilde T } \vert \alpha_n(\vect u) - \alpha_n(\vect v)\vert > (2^{-k} \delta_ n)^\mu 2^{-\mu - 1}\eps\Big), 
\end{multline*} 
where $\tilde T$ denotes a finite set of cardinality $O(n^d)$ and where, for any $\eta > \bar \eta$,
\begin{multline*}
\Prob(\vert S_1 \vert  > (2^{-k}\delta_n)^\mu2^{-\mu -1} \eps)  \\
\leq
\const \times \bigg [ \frac{\int_0^\eta \eps^{-\frac{d}{2\gamma  L } } \, \mathrm{d}\eps + \{ (2^{-k } \delta_n )^{\gamma} + 4n^{-\gamma}\} \eta^{-\frac{d}{\gamma L} } }{ (2^{-k}\delta_n)^\mu2^{-\mu -1} \eps} \bigg ]^{2L}.
\end{multline*}
Set $\eta= 2 (2^{-k}\delta_n)^{\gamma/(1+\frac{d}{\gamma2 L})}$, choose $\gamma$ and $L$ such that $d<2\gamma L$ and note that $4n^{-\gamma} \le 4(2^{-k}\delta_n)^\gamma$.  Then
\begin{align*}
 \Prob(\vert S_1 \vert  > (2^{-k}\delta_n)^\mu2^{-\mu -1} \eps) 
 \leq 
\const \times \Big \{ (2^{-k}\delta_n)^{-\mu + \gamma \frac{1-\frac{d}{2\gamma L}}{1+ \frac{d}{2\gamma L}}} \Big \}^{2L},
\end{align*}
where the constant may depend on $\eps, \gamma, \mu, d, L$.
Therefore, choosing $L$ and $ \gamma$ sufficiently large, we obtain that 
\[ \Prob(\vert S_1 \vert  > (2^{-k}\delta_n)^\mu2^{-\mu -1} \eps) \leq \const \times (2^{-k}\delta_n)^{ \kappa}
\]
 for some $\kappa >0$.

Furthermore, \eqref{smallvalues} and the fact that $2^{-k} \delta_n \ge n^{-1}$ implies that
\begin{multline*}
\Prob \Big(2 \sup_{\vert \vect u - \vect v \vert \leq n^{-1} , \vect u \in \tilde T } \vert \alpha_n(\vect u) - \alpha_n(\vect v) \vert > (2^{-k} \delta_ n)^\mu 2^{-\mu - 1} \eps \Big) \\
 \leq \const \times n^{-\bar \beta} + 3 \ind ( n^{\mu - 1/2 } > \const)
\end{multline*}
for some $\bar \beta>0$, by choosing $\beta \in (2\mu,1)$ and $\ell$ sufficiently large.  Therefore,
 \begin{multline*}
\sum_{k: n^{-1}< 2^{-k} \delta_n} \Prob \Big ( \sup_{\vert \vect u - \vect v \vert \leq 2^{-k}\delta_n} \vert \alpha_n(\vect u) - \alpha_n(\vect v) \vert > \eps (2^{-k}\delta_n)^\mu 2^{-\mu} \Big)\\
 \leq \const \left\{\log(n) \{n^{- \bar \beta} +3 \ind (n^{\mu -1/2}> \const)\} + \delta_n^{\kappa} \sum_{k=0}^\infty 2^{-k \kappa} \right\} =o(1),
\end{multline*}
where the logarithmic term is due to the fact that there are at most $O(\log n)$ summands such that $(2^{-k}\delta_n) > n^{-1}$. The last display is exactly \eqref{eq:sumksum}.

Finally, it remains to be shown that \eqref{eq:condmoment} is satisfied. For $i=1, \dots , n$, let $A_i(\vect u , \vect v) = \ind(\vect U_i  \leq \vect u) - \ind(\vect U_i \leq \vect v) - \{C(\vect u )- C(\vect v)\}$.
Then, by Theorem~2.3.2 in \cite{Bri75}, 
 \begin{multline*}
\mathbb{E}\{\alpha_n(\vect u) - \alpha_n(\vect v)\}^{2L} 
= 
n^{-L} \mathbb{E}\Big \{ \sum_{i=1}^n A_i(\vect u , \vect v ) \Big  \}^{2L} \\
= 
n^{-L} \cum \Big( \prod_{j=1}^{2L } \sum_{i=1}^n A_i(\vect u , \vect v) \Big) \\
= n^{-L}\sum_{\nu_1, \dots ,\nu_R} \prod_{r=1}^R \cum\Big ( \sum_{i=1}^n A_i(\vect u , \vect v), j \in \nu_r\Big),
\end{multline*}
where the sum runs over all partitions of the set $\{ 1, \dots , 2L \}$ and where $\cum(Y_j , j \in \nu)$ denotes the joint cumulant of all  random variables $Y_j$ with $j \in \nu$. Note that, for $\nu_r$ with $\vert \nu_r \vert =1$, we have $\cum\big ( \sum_{i=1}^n A_i(\vect u , \vect v), j \in \nu_r\big)= \mathbb{E}\sum_{i=1}^n A_i (\vect u , \vect v) = 0$, whence it is sufficient to consider $R \leq L$. In that case, an application of Lemma~\ref{lem:kleyvoly} implies that
there exist constants $0<C, C'<\infty$ such that 
\begin{align*}
\cum\Big ( \sum_{i=1}^n A_i(\vect u , \vect v), j \in \nu_r\Big)
&\leq 
C (n+1) \vert \vect u - \vect v \vert (1+ \big \vert \log\vert \vect u - \vect v \vert \big \vert  )^{C'} \\
&\leq 
\bar K (n+1)\vert \vect u -\vect v\vert^{2\gamma}.
\end{align*}
Hence, for any $\vect u , \vect v \in [0,1]^d$ such that $\vert \vect u - \vect v \vert > n^{-1}$,
\begin{multline*}
\mathbb{E}\{\alpha_n(\vect u) - \alpha_n(\vect v)\}^{2L} 
\leq 
\const \times  \sum_{\nu_1, \dots ,\nu_R, R \le L} (n+1)^{R-L} \vert \vect  u - \vect v \vert^{2R \gamma}  \\
\leq 
\const \times \vert  \vect  u - \vect v \vert ^{2L \gamma},
\end{multline*}
which is exactly \eqref{eq:condmoment}.
\end{proof}

\subsection{Proof of Theorem~\ref{theo:weight}} \label{subsec:proofweight}
Throughout the proof, we will use the following additional notations.
Set
\[
\ec(\vect u) 
= G_n \{ \vect G_n^-(\vect u) \}, \qquad \vect G_n^-(\vect u) = \left( G_{n1}^{-}(u_1), \dots, G_{nd}^{-}(u_d) \right)
\]
and define a version of the empirical copula process based on $C_n$ by
\[
\vect u \mapsto \Cb_n(\vect u) = \sqrt n \{ \ec(\vect u) - C(\vect u) \}.
\]
Moreover,  for $0<a<b<1/2$, define
\[
	N(a,b) = \{ \vect u \in [0,1]^d \vert a <  g_1(\vect u)\leq b\}.
\]
Note that $[0,1]^d = \{ \vect u: g_1(\vect u)=0\} \cup N(0,a) \cup N(a,1/2)$. The set $N(a,1/2)$ consists of those vectors such that all of their coordinates are larger than $a$ and such that at most $d-2$ coordinates are larger than or equal to $1-a$. In particular, for $d=2$, we have $N(a,1/2)=(a,1-a)^2$.

The proof of Theorem~\ref{theo:weight} will be based on the following sequence of  Lemmas. All convergences are with respect to $n\to\infty$.

\begin{lemma} \label{lem:cnalt}
Under the conditions of Theorem~\ref{theo:weight},
\[
\sup_{\vect u \in N(\bbb n^{-1}, 1/2)}
\Big \vert \frac{ \hat \Cb_n(\vect u)}{g_\omega(\vect u)} - \frac{  \Cb_n(\vect u) } { g_\omega(\vect u) } \Big \vert = o_P(1).
\]
\end{lemma}

\begin{lemma} \label{lem:epsdiff}
Under the conditions of Theorem~\ref{theo:weight},
\[
\sup_{\vect u \in N(n^{-1/2}, 1/2)}
 \Big \vert \frac{ \Cb_n(\vect u)}{g_\omega(\vect u)} - \frac{ \bar \Cb_n(\vect u) } { g_\omega(\vect u) } \Big \vert = o_P(1).
\]
\end{lemma}

\begin{lemma} \label{lem:cbound}
Under the conditions of Theorem~\ref{theo:weight}, for any $\delta_n \downarrow 0$ such that $\delta_n \ge \bbb n^{-1}$, 
\[ 
\sup_{\vect u \in N(\bbb n^{-1}, \delta_n)} \Big \vert  \frac{\Cb_n (\vect u )}{g_\omega (\vect u)} \Big  \vert  = o_P(1).
\]

\end{lemma}

\begin{lemma} \label{lem:cbound2}
Under the conditions of Theorem~\ref{theo:weight}, for any $\delta_n \downarrow 0$, 
\[ 
\sup_{\vect u \in N(0, \delta_n)} \Big \vert  \frac{\bar \Cb_n (\vect u )}{g_\omega (\vect u)} \Big  \vert  = o_P(1).
\]
\end{lemma}

\begin{lemma} \label{lem:contc}
Under the conditions of Theorem~\ref{theo:weight}, for any $\delta_n \downarrow 0$ 
\begin{align} \label{eq:cbtight}
\sup_{\vect u, \vect u' \in [\frac{\bbb}{n}, 1- \frac{\bbb}{n}]^d :  \vert  \vect u - \vect u' \vert \leq \delta_n} \Big \vert \frac{\Cb_n (\vect u)}{g_\omega (\vect u)} - \frac{\Cb_n (\vect u')}{g_\omega ( \vect u')} \Big \vert =o_P(1)
\end{align}
and 
\begin{align} \label{eq:barcbtight}
\sup_{\vect u, \vect u' \in [0, 1]^d :  \vert  \vect u - \vect u' \vert \leq \delta_n}  \Big \vert \frac{\bar \Cb_n (\vect u)}{\tilde g_\omega (\vect u)} - \frac{\bar \Cb_n (\vect u')}{\tilde g_\omega ( \vect u')} \Big \vert =o_P(1)
\end{align}
\end{lemma}

\begin{proof}[Proof of Theorem~\ref{theo:weight}]
Set $\delta_n  = d n^{-1/2}$. 
Given $\vect u \in [\tfrac{\bbb}{n},1-\tfrac{\bbb}{n}]^d$, choose $\vect u' \in [\frac{1}{\sqrt n} ,1- \frac{1}{\sqrt n} ]^d$  such that $|\vect u - \vect u'| \le \delta_n$. Since
\begin{multline*}
\Big \vert \frac{ \hat \Cb_n(\vect u)}{g_\omega(\vect u)} - \frac{ \bar \Cb_n(\vect u) } { g_\omega(\vect u) } \Big \vert 
\le
\Big \vert \frac{ \hat \Cb_n(\vect u)}{g_\omega(\vect u)} - \frac{  \Cb_n(\vect u) } { g_\omega(\vect u) } \Big \vert 
+
\Big \vert \frac{ \Cb_n(\vect u)}{g_\omega(\vect u)} - \frac{  \Cb_n(\vect u') } { g_\omega(\vect u') } \Big \vert \\
+
\Big \vert \frac{ \Cb_n(\vect u')}{g_\omega(\vect u')} - \frac{ \bar \Cb_n(\vect u') } { g_\omega(\vect u') } \Big \vert  
+
\Big \vert \frac{ \bar \Cb_n(\vect u')}{g_\omega(\vect u')} - \frac{ \bar \Cb_n(\vect u) } { g_\omega(\vect u) } \Big \vert 
\end{multline*}
the first assertion of the theorem follows from Lemma~\ref{lem:cnalt},~\ref{lem:epsdiff} and~\ref{lem:contc}.

Next, let us show that $\bar \Cb_n /\tilde g_\omega \weak \Cb_C / \tilde g_\omega$ in $(\ell^\infty([0,1]^d), \Vert \cdot \Vert_\infty)$.
From Problem 2.1.5 in \cite{VanWel96} and Lemma \ref{lem:contc} we obtain that $\bar \Cb_n/\tilde g_\omega$ is asymptotically equicontinuous. Furthermore, Condition \ref{cond:weak} yields that the finite dimensional distributions of $\bar \Cb_n / \tilde g_\omega$ converge weakly to the finite dimensional distributions of $\Cb_C/\tilde g_\omega$. Note that $\Cb_C/ \tilde g_\omega(\vect u)=\bar \Cb_n /\tilde g_\omega(\vect u)=0$ for any $\vect u$ with at least one entry equal to $0$ or with $d-1$ entries equal to~$1$.
\end{proof}

\begin{proof}[Proof of Lemma~\ref{lem:cnalt}]
It suffices to show that, there exists $\mu \in (\omega, \theta_1)$ such that 
\[
\sup_{\vect u \in [0,1]^d}\vert \hat C_n (\vect u) - C_n(\vect u) \vert = o_P(n^{-1/2-\mu}).
\]
Note that $F_{nj}(X_{ij}) = G_{nj}(U_j)$, whence
\begin{align*}
&\hspace{-0.7cm} \sup_{\vect u \in [0,1]^d} \vert \hat C_n (\vect u) - C_n (\vect u) \vert  \\
& \leq 
\sup_{\vect  u \in [0,1]^d} \left \vert \frac{1}{n} \sum_{i=1}^n \ind\{ \vect G_n(\vect U_i) \leq \tfrac{n+1}{n} \vect u \} - \ind\{ \vect G_n (\vect U_i) \leq \vect u\} \right \vert \\
& \qquad + \sup_{\vect u \in [0,1]^d } \left \vert\frac{1}{n} \sum_{i = 1}^n \ind \{ \vect G_n (\vect U_i) \leq \vect u \} - \ind\{ \vect U_i \leq \vect G_n^{-}(\vect u) \} \right \vert \\
&\leq  \sum_{\dimi =1}^d \left [ \sup_{u \in[0,1] }  \frac{1}{n} \sum_{i=1}^n\ind \{ u < G_{n\dimi}(U_{i \dimi} )\leq \tfrac{n+1}{n}u\} \right. \\
& \qquad + \left. \sup_{u \in [0,1] } \frac{1}{n} \sum_{i=1}^n \left \vert \ind\{ G_{n \dimi} (U_{i \dimi} ) \leq u \} - \ind\{ 
U_{i \dimi} \leq G_{n \dimi}^{-}(u) \} \right \vert \right ] 
\end{align*}
From the definition of the empirical distribution function and the generalized inverse function we have that, for any fixed $u$, both $\sum_{i=1}^n \ind\{ u < G_{n \dimi}(U_{i \dimi}) \leq \frac{n+1}{n}u\}$ and $\sum_{i=1}^n  \vert \ind\{ G_{n\dimi}(U_{i \dimi})  \leq u \} - \ind\{ U_{i \dimi}\leq  F_{n \dimi}^{-}(u) \} \vert$ are bounded by the maximal number of $U_{ij}$ which are equal.
Note that this maximal number is equal to 
$n \times \sup_{u \in [0,1]} \vert G_{n\dimi}(u) - G_{n \dimi}(u-)\vert$. Provided there are no ties among $U_{1j}, \dots, U_{nj}$, for any $j=1, \dots, d$ (which, for instance, occurs in the i.i.d.\ case), this expression is equal to $1$ and the Lemma is proven. In the general case, we have
\begin{align}
\sup_{u \in [0,1]} \vert G_{n\dimi}(u) - G_{n \dimi}(u-)\vert
&\leq \sup_{\substack{u,v \in [0,1] \\ \vert u - v \vert \leq 1/n}} \vert G_{n \dimi}(u ) - G_{n \dimi}(v) \vert \nonumber\\
& \leq  \sup_{\substack{u,v \in [0,1] \\ \vert u - v \vert \leq 1/n}} \vert G_{n \dimi}(u ) - G_{n \dimi}(v) -(u-v)\vert + \frac{1}{n} \nonumber\\
& \leq  \frac{1}{\sqrt n} \sup_{\substack{\vect u, \vect v \in [0,1]^d \\ \vert \vect u - \vect  v \vert \leq 1/n}} \vert\alpha_n(\vect u) - \alpha_n(\vect v)\vert + \frac{1}{n} \label{eq:maxjump}
\end{align}
Then, the assertion follows from Condition \ref{cond:modulus}.
\end{proof}

\begin{proof}[Proof of Lemma~\ref{lem:epsdiff}]
First of all, we write
\[
\Cb_n(\vect u) -  \bar \Cb_n (\vect  u ) = ( B_{n1} + B_{n2} + B_{n3}  ) (\vect u)
\]
where
\begin{align*}
B_{n1}(\vect u) &= \alpha_n\{   \vect G_n^{-}(\vect u)  \} - \alpha_n(\vect u) \\
B_{n2}(\vect u) &= \sqrt n \left[ C \{ \vect G_n^{-}(\vect u)  \} - C(\vect u)  \right]
	- \sum_{\dimi=1}^d \dot C_\dimi(\vect u) \beta_{n\dimi}(u_\dimi) \\
B_{n3}(\vect u) &=  \sum_{\dimi=1}^d  \dot C_\dimi(\vect u) \left\{ \beta_{n\dimi}(u_\dimi) + \alpha_{n\dimi} (u_\dimi ) \right\} .
\end{align*}
For $p=1,2,3$, set $A_{np}(\vect u) = B_{np}(\vect u) / g_\omega(\vect u)$. The Lemma is proved if we show uniform negligibility of each term individually.

\medskip

\noindent
{\it Treatment of $A_{n1}$.} Let $\Omega_n$ denote the event that $\sup_{\vect u \in [0,1]^d}|\vect G_n^{-}(\vect u) - \vect u| \le \delta_n = n^{-1/2+ \kappa}$, with $\kappa>0$ to be specified later on. Note that the probability of  $\Omega_n$ converges to $1$.  Exploiting  Condition \ref{cond:modulus} and the fact that  $\vert g_\omega (\vect u )\vert^{-1}  \leq n^{\omega/2} $ for $\vect u \in N(n^{-1/2}, 1/2)$  we obtain, for any $\mu \in (0,\theta_1)$, 
\begin{align*}
\sup_{\vect u \in N(n^{-1/2}, 1/2) }\vert A_{n1} (\vect u ) \vert 
& \le  
n^{\omega/2} \sup_{\vect u \in [0,1]^d }\left|  \alpha_n\{   \vect G_n^{-}(\vect u)  \} - \alpha_n(\vect u)  \right| \\
& \le 
n^{\omega/2} M_n(\delta_n, \mu) \sup_{\vect u \in [0,1]^d }\left\{ |\vect G_n^-(\vect u) - \vect u|^{\mu} \vee n^{-\mu} \right\} \ind_{\Omega_n}  \\
&\hspace{7.5cm}+ o_P(1) \\
& \le
n^{\omega/2- \mu/2 + \kappa \mu} o_P(1) + o_P(1) 
\end{align*}
The right-hand side is $o_P(1)$ if we choose $\mu \in (\omega, \theta_1)$ sufficiently large and $\kappa>0$ sufficiently small such that $\omega< \mu(1-2\kappa)$. 

\medskip

\noindent
{\it Treatment of $A_{n2}$.} Fix $\vect u \in N(n^{-1/2}, 1/2)$. Let $S=S_{\vect u}$ denote the set of all $\dimi \in\{1, \dots, d\}$ such that $u_\dimi \in [n^{-1/2}, 1-n^{-\gamma}]$, with $\gamma>1/2$ to be specified later. Let $(\vect G_n^-(\vect u))_S$ denote the vector in $\R^d$ whose $\dimi$th coordinate is equal to $G_{n\dimi}^-(u_\dimi) \ind(\dimi \in S) + u_\dimi \ind(\dimi \not \in S)$. Write $A_{n2}(\vect u)= D_{n1}(\vect u) + D_{n2}(\vect u)$, where
\begin{align*}
D_{n1}(\vect u) 
&= 
\bigg(  \sqrt n \left[ C\{ \vect G_n^-(\vect u)\} - C\{(\vect G_n^-(\vect u))_S\} \right] - \sum_{\dimi \notin S} \dot C_\dimi(\vect u) \beta_{n\dimi}(u_\dimi) \bigg) g_\omega^{-1}(\vect u),\\
D_{n2}(\vect u) 
&=
\bigg(  \sqrt n \left[ C\{(\vect G_n^-(\vect u))_S\} - C(\vect u) \right] - \sum_{\dimi \in S} \dot C_\dimi(\vect u)\beta_{n\dimi}( u_\dimi)  \bigg) g_\omega^{-1}(\vect u).
\end{align*}
Since $\dot C_\dimi \in [0,1]$, we can bound
\[
D_{n1}(\vect u) 
\le 
2 \sum_{ \dimi \notin S} \left\vert \frac{\beta_{n\dimi}(u_\dimi) }{g_\omega(\vect u)} \right\vert
\le
2 \sum_{\dimi =1}^d \sup_{u_\dimi \in [1-n^{-\gamma}, 1]} \left\vert \frac{\beta_{n\dimi}(u_\dimi) }{n^{-\omega/2}} \right\vert.
\]
The right-hand side is $o_P(1)$ by Lemma~\ref{lem:betabound}.
Regarding $D_{n2}$, by Taylor's Theorem, 
$
\vert D_{n2} (\vect u) \vert 
=
\tfrac{1}{2}\sum_{\dimi_1, \dimi_2 \in S} D_{n2}^{\dimi_1 \dimi_2} (\vect u),
$
where 
\[
D_{n2}^{\dimi_1 \dimi_2}(\vect u)= n^{-1/2}  \ddot C_{\dimi_1\dimi_2 }( {\vect \xi}_{n})   \beta_{n\dimi_1}(u_{\dimi_1}) \beta_{n\dimi_2}(u_{\dimi_2}) g_\omega (\vect u) ^{-1},
\]
and where $\vect  \xi_n =(\xi_{n1}, \dots, \xi_{nd})'$ is an intermediate point between $(\vect G_n^{-}(\vect u))_S$ and $\vect u$. 
By Condition~\ref{cond:2nd}, we have
\[
\vert \ddot C_{\dimi_1 \dimi_2}({\vect  \xi}_n) \vert 
\le 
K  \{\xi_{n\dimi_1} (1- \xi_{n\dimi_1}) \}^{-1/2} \{\xi_{n\dimi_2}(1- \xi_{n\dimi_2}) \}^{-1/2} .
\]
Therefore,   since $g_\omega (\vect u )^{-1} \le n^{\omega/2}$, 
\begin{multline*}
\vert D_{n2}^{\dimi_1 \dimi_2}(\vect u) \vert 
\le
K n^{-1/2+ \omega/2} 
\sup_{\vect u \in [n^{-1/2}, 1-n^{-1/2}]^d}   \left|
\left\{ \frac{  u_{\dimi_1} (1- u_{\dimi_1}) } {\xi_{n\dimi_1} (1- \xi_{n\dimi_1})} \right\}^{1/2} 
\right. \\
\times \left.
\left\{ \frac{  u_{\dimi_2} (1- u_{\dimi_2}) } {\xi_{n\dimi_2} (1- \xi_{n\dimi_2})} \right\}^{1/2} 
\times
\frac{ \vert  \beta_{n\dimi_1}(u_{\dimi_1}) \vert }{\{ u_{\dimi_1}(1- u_{\dimi_1}) \}^{\omega}} 
\times
\frac{\vert  \beta_{n\dimi_2}(u_{\dimi_2}) \vert }{\{ u_{\dimi_2}(1- u_{\dimi_2}) \}^{\omega}}
\right.\\
\times \left.
\left\{ u_{\dimi_1} (1- u_{\dimi_1}) u_{\dimi_2} (1- u_{\dimi_2}) \right\}^{\omega-1/2} 
\right|.
\end{multline*}
By an application of Lemma~\ref{lem:uxi} and by Condition~\ref{cond:quantile}, the right-hand side 
is of order $O_P(n^{-1/2+\omega/2+ \gamma (1-2\omega)})=o_P(1)$, provided we choose $\gamma\in (1/2, \{1/2+\omega/(2-4\omega)\}\wedge\{1/(2(1-\theta_2))\}\wedge\theta_3)$. Since $\vect u\in N(n^{-1/2}, 1/2)$ was arbitrary,  we can conclude that $\sup_{\vect u \in N(n^{-1/2}, 1/2) }\vert A_{n2} (\vect u ) \vert  =o_P(1)$.

\medskip

\noindent
{\it Treatment of $A_{n3}$.} Since $\vert \dot C_j (\vect u) \vert \leq 1$ for any $\vect u \in [0,1]^d$,  we have
\begin{align*}
\sup_{\vect  u \in N(n^{-1/2}, 1/2) } \vert A_{n3}(\vect u) \vert 
&\le
n^{\omega/2} \sum_{\dimi =1}^d \sup_{ u_j \in [0,1]}\vert \beta_{n\dimi}(u_\dimi) + \alpha_{n\dimi} \{ G_{n\dimi}^-( u_\dimi )  \} \vert \\ 
& \hspace{1.4cm} + 
n^{\omega/2} \sum_{\dimi =1}^d \sup_{  u_j \in [0,1 ]}  \left\vert \alpha_{n\dimi} \{ G_{n\dimi}^-( u_\dimi ) \} -   \alpha_{n\dimi}(u_\dimi) \right\vert . 
\end{align*}
The second sum on the right-hand side is of order $o_P(1)$ as shown in the preceding treatment of the term $A_{n1}$. Negligibility of the first sum
follows from Lemma \ref{lem:inverse},
observing that
$\alpha_{n\dimi} \{ G_{n\dimi}^-( u_\dimi )  \} = \sqrt{n}[ G_{n \dimi}\{ G_{n \dimi}^-( u_\dimi) \} -G_{n \dimi}^-(u_\dimi) ]$ from the definition of $\alpha_n$.
\end{proof}

\begin{proof}[Proof of Lemma~\ref{lem:cbound}.] Note that, by a monotonicity argument, it suffices to treat sequences $\delta_n$ such that $\delta_n \gg n^{-1/2}$, i.e., $\delta_n \sqrt n \to \infty$.
First of all, choose $\gamma$ such that
$1/2+\omega < \gamma < 1/\{ 2(1- \theta_2)\}\wedge \theta_3$.  
Set $M_{n\gamma }=N( n^{-\gamma}, \delta_n)\cap (n^{-\gamma} , 1- n^{-\gamma})^d$ and $M_{n\gamma}^c=N( n^{-\gamma}, \delta_n)\setminus (n^{-\gamma}, 1- n^{-\gamma})^d$, and note that
$N(\bbb n^{-1}, \delta_n) = N(\bbb n^{-1}, n^{-\gamma}) \cup M_{n\gamma} \cup M_{n\gamma}^c$.
Therefore,
\begin{align} \label{eq:rn}
\sup_{\vect u \in N(\bbb n^{-1}, \delta_n)} \Big \vert  \frac{\Cb_n (\vect u )}{g_\omega (\vect u)} \Big  \vert  
= 
 R_{n}\{ N(\bbb n^{-1}, n^{-\gamma}) \} \vee R_{n}(M_{n\gamma}) \vee R_n(M_{n\gamma}^c),
\end{align}
where, for $A \subset [0,1]^d$, 
$
R_{n}(A) = \sup_{\vect u \in A}  \vert { \Cb_n (\vect u )}/{g_\omega (\vect u)}  \vert.
$
It suffices to show negligibility of each term on the right-hand side of \eqref{eq:rn}.

\medskip

\noindent
{\it Treatment of $R_{n}\{ N(\bbb n^{-1}, n^{-\gamma}) \}$.} We will distinguish the cases that either $g_\omega(\vect u) =u_1^\omega$ or $g_\omega(\vect u)=(1-u_1)^\omega$. The cases  $g_\omega(\vect u) =u_\dimi^\omega$ or $g_\omega(\vect u)=(1-u_\dimi)^\omega$ for some $\dimi >1$ can be treated similarly.

Let us first consider $\vect u$ such that $g_\omega(\vect u)=u_1^\omega$. Obviously, 
\[ 
 \vert \ec (\vect u ) - C(\vect u) \vert \leq \vert \ec (\vect  u ) - \ec (0, u_2, \dots ,  u_d) \vert  + \vert  C ( 0, u_2, \dots ,  u_d) - C (\vect  u ) \vert.
\]
By Lipschitz-continuity of the copula function $C$, the second term on the right hand side can be bounded by $u_1=g_1(\vect u)$. For  the first term, note that 
\begin{multline}\label{eq:empcop}
\vert \ec (\vect  u ) - \ec ( 0, u_2, \dots ,  u_d) \vert 
= 
\frac{1}{n} \sum_{i=1}^n \ind \{\vect U_i \leq \vect G_n^{-} (\vect u)\} \\
\leq 
\frac{1}{n} \sum_{i=1}^n \ind \{ U_{i1} \leq G_{n1}^{-}(u_1) \}  
=
G_{n1}\{ G_{n1}^-(u_1) \} 
\end{multline}
By Lemma \ref{lem:inverse} the last expression is equal to  $u_1 + o_P(n^{-1/2-\mu})=g_1(\vect u) +o_P(n^{-1/2-\mu})$ for any $\mu \in(\omega, \theta_1)$, where the residual term is uniformly in $u_1 \in [0,1]$. Combined, this yields
$
\vert \Cb_n(\vect u) \vert \leq \sqrt n 2 g_1(\vect u ) + o_P(n^{-\mu}),
$
and hence
\[
\sup_{\vect u \in N(\bbb n^{-1}, n^{-\gamma}), g_\omega(\vect u) = u_1^\omega} \Big \vert  \frac{\Cb_n (\vect u )}{g_\omega (\vect u)} \Big  \vert 
\le
 2n^{1/2+\omega-\gamma} + o_P(n^{-\mu+\omega}) =o_P(1).
\]

Now, consider the case $g_\omega(\vect u)=(1 - u_1)^\omega$, i.e., $1-u_1=1-u_1 \wedge \dots \wedge \widehat{u_k} \wedge \dots \wedge u_d$ for some $ k \in \{ 2, \dots, d \}$ and without loss of generality we may assume that $k=2$. Then, in particular, $1-u_1 \le 1-u_2$ and $1-u_1 \ge 1-u_j$ for all $j \ge 3$.
Now, decompose
\begin{multline*}
 \vert \ec (\vect u ) - C(\vect u) \vert \leq \vert \ec (\vect  u ) - \ec (\vect u ^{(2)}) \vert 
+ \vert \ec ( \vect u ^{(2)})  -  C ( \vect u ^{(2)})  \vert 
 + \vert  C ( \vect u ^{(2)}) - C (\vect  u ) \vert  .
\end{multline*}
Again by Lipschitz-continuity of the copula function, we have 
\[
\vert  C ( \vect u ^{(2)}) - C (\vect  u ) \vert \leq \sum_{ \dimi \neq 2} \vert 1- u_\dimi \vert \leq (d-1) \vert 1-u_1 \vert = (d-1)g_1(\vect u).
\]
Furthermore, we have
\begin{multline} \label{eq:empcop2}
\vert \ec (\vect  u ) - \ec (\vect u ^{(2)}) \vert\leq  \vert \ec(\vect u ) - \ec\{1, u_2, \dots , u_d\} \vert\\
 + \vert \ec \{1, u_2, \dots , u_d\}- \ec\{1,u_2,1, u_4, \dots , u_d\}\vert  \\
+ \dots + \vert \ec \{1,u_2,1, 1, \dots ,1, u_d\}-  \ec(\vect u^{(2)}) \vert
\end{multline}
and thus, by similar arguments as in \eqref{eq:empcop}, $\vert \ec (\vect  u ) - \ec (\vect u ^{(2)}) \vert\leq  (d-1) g_1(\vect u) + o_P(n^{-1/2-\mu})$, uniformly in $\vect u$. Finally, from Lemma \ref{lem:inverse}
\[ 
\vert \ec ( \vect u ^{(2)})  -  C ( \vect u ^{(2)})  \vert  = \vert G_{n2}\{G_{n2}^{-}(u_2 ) \} - u_2 \vert = o_P(n^{-1/2-\mu}).
\]
Altogether, we obtain
\[
\sup_{\vect u \in N(\bbb n^{-1}, n^{-\gamma}), g_\omega(\vect u) = (1-u_1)^\omega} \Big \vert  \frac{\Cb_n (\vect u )}{g_\omega (\vect u)} \Big  \vert 
\le
 2(d-1) n^{1/2+\omega-\gamma} + o_P(n^{-\mu+\omega}) =o_P(1).
\]

\medskip
\noindent
{\it Treatment of $R_{n}(M_{n\gamma})$.}
Again, let us first treat the case where $g_\omega(\vect u )= u_1^\omega$. We can write
$
 {\Cb_n (\vect u)}/{g_\omega (\vect u)}  = S_{1n}(\vect u) + S_{2n}(\vect u)+S_{3n} (\vect u),
$
where
\begin{align*}
S_{1n} (\vect u) &= {\sqrt n [ G_n \{ \vect G_n^- (\vect u ) \} - C \{ \vect G_n^- (\vect u)\} ]} / {g_\omega (\vect u )} \\
S_{2n} (\vect u) &= {\sqrt n [ C\{ \vect G_n^- (\vect u )\} - C \{ G_{n1}^- (u_1), u_2 , \dots ,u_d \} ]} / {g_\omega(\vect u)}\\
S_{3n} (\vect u) &= {\sqrt n [ C \{ G_{n1}^- (u_1), u_2 , \dots ,u_d \} - C(\vect u ) ]} / {g_\omega (\vect u )}.
\end{align*}
Lipschitz continuity of the copula $C$ together with Condition~\ref{cond:quantile} implies that $\sup_{\vect u \in M_{n \gamma } , g_\omega(\vect u)=u_1^\omega}\vert S_{3n}(\vect u) \vert =o_P(1)$.

Regarding $S_{1n}$, let $\Omega_n$ denote the event that $\sup_{u_1 \in [0,\delta_n]} G_{n1}^-(u_1) \le 2 \delta_n$. On $\Omega_n^c$, we have
$
\sqrt n\delta_n < \sup_{u_1 \in [0,\delta_n]} \sqrt n| G_{n1}^-(u_1) - u_1| =O_P(1) $,
whence, by the assumption that
$\sqrt n \delta_n \to \infty$, we get $\Pr(\Omega_n^c) \to 0$. Therefore, by Condition~\ref{cond:modulus}, for any $\mu\in (0,\theta_1)$, we have 
\begin{align*}
\vert S_{1n } (\vect u ) \vert 
&= 
\Big  \vert \frac{\alpha_n\{ \vect G_n^- (\vect u) \}  - \alpha_n \{0, G_{n2}^-(u_2), \dots , G_{nd}^-(u_d) \} }{u_1^\omega} \Big \vert  \\
&\le
M_n(2\delta_n, \mu) \left| \frac{\{ G_{n1}^{-}(u_1) \}^\mu \vee n^{-\mu}}{u_1^\omega} \right| \ind_{\Omega_n} + o_P(1) \\
& \leq  
o_P(1)  \left\{ 
\frac{\vert G_{n1}^-(u_1)- u_1 \vert^\mu}{u_1^\omega} +  u_1^{\mu - \omega} \right\} \vee n^{-\mu+ \gamma \omega}  + o_P(1),
\end{align*}
where we used subadditivity of the function $x \mapsto x^\mu$, $x\ge 0$. By Condition~\ref{cond:quantile}, we have
\begin{multline*}
\sup_{u_1 \in [n^{- \gamma}, \delta_n] }\frac{\vert G_{n1}^-(u_1)- u_1 \vert^\mu}{u_1^\omega} 
\le
n^{-\mu/2} \sup_{u_1 \in [n^{- \gamma}, \delta_n] } | u_1^{ \omega (\mu-1)}  | O_P(1)\\
 =O_P(n^{-\mu/2 -   \gamma \omega (\mu-1)} )
\end{multline*}
%
%
Exploit that $ \gamma < 1$ and choose $\mu \in (\omega/( \omega +1/2), \theta_1)$ to obtain that, as $n \to \infty$,
$\sup_{\vect u \in M_{n \gamma }, g_\omega(\vect u)=u_1^\omega}\vert S_{1n}(\vect u) \vert =o_P(1)$.

Finally, we turn to $S_{2n}$. The mean value theorem allows to write
\[
S_{2n}(\vect u)
=
\sum_{j=2}^d
\frac{ \dot C_j \{ G_{n1}^-(u_1), \zeta_2, \dots, \zeta_d\} \sqrt n\{ G_{nj}^-(u_j)-u_j\} }{g_\omega(\vect u)}
=:
\sum_{j=2}^d S_{2nj}(\vect u)
\]
for some intermediate values $\zeta_j$ between $u_j$ and $G_{nj}^-(u_j)$, for $j=2, \dots, d$. We may consider each summand individually; let us fix $j\in\{2, \dots, d\}$  and distinguish two cases. First, suppose that $1-u_j < u_1=g_1(\vect u)$. Then, with $\omega'\in(\omega, \theta_1)$,
\[
|S_{2nj}(\vect u) | 
\le
\frac{ \sqrt n | G_{nj}^-(u_j)-u_j|  }{(1-u_j)^{\omega'}} (1-u_j)^{\omega'-\omega} =o_P(1),
\]
by Condition~\ref{cond:quantile} and the fact that $n^{-\gamma}< (1-u_j) \le \delta_n$. 
Now, suppose that $1-u_j \ge  u_1=g_1(\vect u)> n^{-\gamma}$.
Since  $\dot C_j (0,u_2, \dots , u_d)=0$ for any $j=2, \dots ,d$, another application of the mean value theorem allows to write
\begin{align*}
S_{2nj} (\vect  u )  
= 
\frac{\ddot C_{j1} ( \vect \xi_j) G_{n1}^-(u_1) \sqrt n \{ G_{n\dimi}^-(u_\dimi)-u_\dimi \}}{u_1^\omega} ,
\end{align*}
where $\vect \xi_j=(\xi_{j1},\zeta_2, \dots, \zeta_d)$ satisfies $\xi_{j1}\in (0, G_{n1}^-(u_1) )$.
Now, fix $\omega'  \in (0,\theta_2)$ such that $\omega'>(1- \frac{1}{2  \gamma}) \vee \omega$. By Condition~\ref{cond:2nd} and Lemma~\ref{lem:uxi}, we have
\begin{align} \label{eq:sn2}
\vert S_{2nj} (\vect u ) \vert 
& \le
\frac{G_{n1}^-(u_1)}{u_1^\omega}  \left \vert \frac{\sqrt n \{ G_{nj}^- (u_j) - u_j  \} }{\{ u_j (1- u_j) \}^{\omega'}} \right \vert 
\times K  \frac{\{ u_j (1- u_j) \}^{\omega'} }{\xi_{jj} (1- \xi_{jj})}  \\
& \le
\left\{  n^{-1/2}\frac{\sqrt n | G_{n1}^-(u_1)-u_1 |}{u_1^\omega} + u_1^{1-\omega} \right\}  \{u_j(1-u_j) \}^{\omega'-1} O_P(1) \nonumber
\end{align}
Observing that $u_j \ge u_1$ as a consequence of $g_\omega(\vect u) = u_1^\omega$ and that $1-u_j \ge  u_1$ by assumption, we obtain
\begin{align*}
\{u_j(1-u_j) \}^{\omega'-1} \le [ \{u_j \wedge (1-u_j) \}/2 ]^{\omega'-1} \le 2^{1-\omega'} u_1^{\omega'-1} \le 2 u_1^{\omega'-1},
\end{align*}
where we used the fact that $u(1-u) \ge \{ u \wedge (1-u) \}/2$ for all $u\in[0,1]$. Therefore, we can bound the right-hand side of \eqref{eq:sn2} by
\[
\left\{ n^{-1/2}u_1^{\omega' - 1 } O_P(1) + u_1^{\omega' - \omega} \right\} \times O_P(1),
\]
where all $O_P$-terms are uniform in $\{\vect u \in M_{n\gamma}: g_\omega(\vect u)=u_1^\omega\}$. 
Thus, by the choice of $\gamma$ and $\omega'$, $\sup_{\vect u \in M_{n\gamma}, g_\omega(\vect u)=u_1^\omega}\vert S_{2n}(\vect u) \vert =o_P(1)$.

For the treatment of $R_{n}(M_{n\gamma})$, it remains to consider the case $g_\omega (\vect u) =(1-u_1)^\omega$, i.e., $1-u_1=1-(u_1 \wedge \dots \wedge \widehat{u_k} \wedge \dots \wedge u_d)$ for some $ k \in \{ 2, \dots, d \}$. 
Again, without loss of generality, we may assume that $k=2$, which implies that $1-u_1 \le 1-u_2$ and $1-u_1 \ge 1-u_j$ for all $j \ge 3$. Note that, additionally, $1-u_\dimi> n^{-\gamma}$ for all $\dimi=1, \dots ,d$ since $\vect u \in M_{n\gamma}$.
Now, 
\begin{align*}
\frac{\Cb_n(\vect u)}{g_\omega (\vect u)} 
= 
\frac{\alpha_n\{ \vect G_n^-(\vect u) \} + \sqrt n [C\{ \vect G_n^-(\vect u) \}- C(\vect u)]}{g_\omega(\vect u)}
= 
\sum_{p=1}^4 T_{pn}(\vect u) 
\end{align*}
with 
\begin{align*}
T_{1n} (\vect u) 
&= 
\frac{\alpha_n \{ \vect G_n^- (\vect u)  \} - \alpha_n \{ 1, G_{n2}^- (u_2),1, \dots ,1\}}{g_\omega (\vect u)}\\
T_{2n}( \vect u) 
&= 
\frac{\alpha_n\{ 1, G_{n2}^- (u_2),1, \dots ,1\} + \sqrt n \{ G_{n2}^- (u_2) - u_2 \}}{g_\omega (\vect u )}\\
T_{3n }(\vect u) 
&= 
\frac{\sqrt n [ C \{ \vect G_n^- (\vect u ) \} - C\{ G_{n1}^- (u_1), u_2 , G_{n3}^- (u_3) , \dots , G_{nd}^- (u_d) \}] }{g_\omega(\vect u)}\\
& \hspace{7cm} - \frac{\sqrt n \{ G_{n2}^-(u_2)- u_2 \} }{g_\omega (\vect u )}\\
T_{4n} (\vect u) 
& = 
\frac{\sqrt n [ C\{ G_{n1}^- (u_1), u_2 , G_{n3}^- (u_3) , \dots , G_{nd}^- (u_d) \} - C (\vect u)] }{g_\omega (\vect u)}
\end{align*} 

Concerning $T_{1n}$,  we can proceed similar as for $S_{1n}$ above. Define the event $\Omega_n$ by $| \vect G_n^- (\vect u) - (1, G_{n2}^- (u_2),1, \dots ,1)' | \le 2d \delta_n$ and note that $\Prob(\Omega_n^c) \to 0$. Then, by Condition~\ref{cond:modulus} applied with $\mu \in (\omega/(\omega+\frac{1}{2}),\theta_1)$,
\begin{align*}
\vert T_{1n}(\vect u) \vert 
\leq  
M_n(2d\delta_n, \mu)
\frac{| \vect G_n^- (\vect u) - (1, G_{n2}^- (u_2),1, \dots ,1)' |^\mu \vee n^{-\mu}}{(1-u_1)^\omega} \ind_{\Omega_n} + o_P(1).
\end{align*}
Use the fact that $\gamma<1$ and $1-u_1\ge 1-u_j\ge n^{-\gamma}$ for $j\ge 3$ and subadditivity of $x \mapsto x^\mu$ to bound the right-hand side by
\begin{multline*}
O_P(1) \sum_{\dimi \neq 2} \frac{\vert G_{n \dimi}^- (u_\dimi) - u_\dimi \vert ^\mu + \vert 1- u_\dimi \vert ^\mu}{(1-u_j)^\omega}  + o_P(1)\\
\leq 
O_P(1) \left\{ \sum_{\dimi \neq 2} n^{- \mu/2 + \omega- \omega \mu} \left\{ \frac{\sqrt n \vert G_{n \dimi }^- (u_\dimi) - u_\dimi \vert }{(1-u_\dimi)^\omega} \right\}^\mu + \delta_n^{\mu - \omega} \right\} + o_P(1).
\end{multline*}
Therefore, by Condition~\ref{cond:quantile} and by the choice of $\mu$, $\vert T_{1n}(\vect u) \vert =o_P(1)$ uniformly in $\{ \vect u \in M_{n\gamma}:g_\omega(\vect u) = (1- u_1)^\omega\}$.

Regarding $T_{2n}$, by the definition of $\alpha_n$ and since $g_1(\vect u) = 1-u_1 \geq n^{-1}$,
\[
\sup_{\vect u \in M_{n\gamma}, g_1(\vect u) = 1-u_1 } \vert T_{2n}(\vect u) \vert 
\leq      
n^\omega \sup_{u_2 \in [0,1]}\sqrt n  \vert G_{n2}\{ G_{n2}^- ( u_2) \} -u_2 \vert.
\]  
An application of Lemma \ref{lem:inverse} with $\mu \in (\omega , \theta_1)$ yields that the right-hand side is of order $o_P(n^{-\mu+\omega})=o_P(1)$. 

Regarding $T_{3n}$, choose $\omega'  \in (\omega \vee (1- \frac{1}{2 \gamma}), \theta_2)$. By the mean-value theorem, we can write
\begin{align*}
T_{3n} (\vect u) 
=  \frac{\sqrt n [ \dot C_2\{ G_{n1}^- (u_1) ,\zeta_2, G_{n3}^-(u_3), \dots , G_{nd}^-(u_d) \} - 1 ] \{ G_{n2}^- (u_2) - u_2 \}}{g_\omega (\vect u)} 
\end{align*}
for some intermediate value $\zeta_2$ between $G_{n2}^-(u_2)$ and $u_2$.
Due to the fact that $\dot C_2\{ 1,\zeta_2, 1, \dots ,1\}=1$, a second application of the mean value theorem allows to write the right-hand side of the last display as
\begin{align*}
T_{3n}(\vect u) = \sum_{\dimi\neq 2}  \frac{\sqrt n  \ddot C_{2\dimi}(\vect \xi) \{ G_{n2}^- (u_2) - u_2 \}\{ G_{n\dimi}^- (u_\dimi) - 1 \}}{g_\omega (\vect u)} 
\end{align*}
for some $\vect \xi$ lying between $\vect G_n^- (\vect u)$ and $\vect u$. Hence, by Condition~\ref{cond:2nd}, Condition~\ref{cond:quantile} and Lemma~\ref{lem:uxi}, we can bound $T_{3n}$ as follows:
\begin{align*}
\vert T_{3n}(\vect u)  \vert|
&\leq  
\frac{\sqrt n \vert G_{n2}^-(u_2) - u_2 \vert }{\{ u_2 (1- u_2) \}^{\omega'}} \frac{\{ u_2 (1- u_2) \}^{\omega'}}{ (1- u_1) ^{\omega}} \frac{O_P(1) }{u_2 (1-u_2)} \sum_{\dimi \neq 2} \vert G_{n \dimi} ^- (u_\dimi) -1 \vert \\
&= 
O_P(1) \{ u_2 (1- u_2 ) \}^{\omega ' -1 } \sum_{\dimi \neq 2}\bigg \{\frac{ \vert G_{n \dimi }^- (u_\dimi) - u _\dimi \vert}{(1-u_\dimi)^\omega} + (1-u_j)^{1-\omega}\bigg \}.
\end{align*}
Since $1-u_2 \ge 1-u_1$ and $u_2 \ge 1-u_1$, the right-hand side is of order
$O_P\{ n^{-1/2} (1-u_1)^{\omega'-1} + (1-u_1)^{\omega' - \omega} \}
= o_P(1)$
uniformly in $\vect u \in M_{n\gamma}$ such that $g_\omega (\vect u) =(1- u_1)^\omega$, by the choice of $\omega'$.

Finally, regarding $T_{4n}$, Lipschitz-continuity of the copula function and Condition~\ref{cond:quantile} immediately imply that for any $\omega' \in (\omega, \theta_2)$
\[
\vert T_{4n}(\vect u) \vert 
\le 
\sum_{\dimi \neq 2} (1-u_1)^{-\omega} (1-u_\dimi)^{\omega'}  \frac{\sqrt n \vert G_{n \dimi }^- (u_\dimi) - u_\dimi \vert }{(1-u_\dimi)^{\omega'}}
=
O_P((1-u_1)^{\omega'-\omega}),
\]
which is of order $o_P(1)$ uniformly in $\{\vect u \in M_{n\gamma}:g_\omega(\vect u) = (1- u_1)^\omega\}$.

\medskip
\noindent
{\it Treatment of $R_{n}(M_{n\gamma}^c)$.}
First note that, from the definition of  $N(n^{-\gamma},\delta_n)$, for every $\vect u \in M_{n \gamma }^c$ there are at most $d-2$ components larger than or equal to $1-n^{-\gamma}$. For that reason, we can write
\begin{align*} 
M_{n \gamma }^c = \bigcup_{\vect \ell =(\ell_1, \dots ,\ell_d )\in \{ 0,1 \}^d ; \vert \vect \ell \vert \geq 2} S_{\ell_1}\times \dots \times S_{\ell_d},
\end{align*}
where  $\vert \vect \ell \vert = \sum_{\dimi =1}^d \ell_\dimi $, $S_0=[1-n^{-\gamma},1]$ and $S_1=(n^{-\gamma}, 1-n^{- \gamma})$. 
In order to show negligibility of $R_{n}(M_{n\gamma}^c)$, it suffices to fix a vector $\ell$ with $|\ell|\ge 2$ and to show uniform negligibility of $\Cb_n/g_\omega$ over $\vect u \in S_{\vect \ell}:=S_{\ell_1} \times \dots \times S_{\ell_d} $.

For $\vect u \in [0,1]^d$, let $\vect u^{(\vect \ell)}$ denote the vector 
whose $\dimi$th component (with $\dimi =1,...,d$) is equal to $\ind(\ell_\dimi=0) + u_\dimi \ind(\ell_\dimi=1)$. 
Then,
\begin{align*}
\sup_{\vect u \in S_{\vect \ell }}\Big \vert \frac{\Cb_n(\vect u)}{g_\omega(\vect u)} \Big \vert 
& \leq  
\sup_{\vect u \in S_{\vect \ell }} \frac{\sqrt n \vert G_n\{\vect G_n^- (\vect u)\}- G_n \{ \vect G_n^- (\vect u)^{(\vect \ell)} \}\vert }{n^{-\omega \gamma}} \\
&\qquad + \sup_{\vect u \in S_{\vect \ell }} \frac{\sqrt n \vert G_n \{ \vect G_n^- (\vect u)^{(\vect \ell)} \} - C (\vect u^{(\vect \ell)})\vert }{g_\omega(\vect u)} \\
&\qquad + \sup_{\vect u \in S_{\vect \ell }} \frac{\sqrt n \vert C(\vect u^{(\vect \ell)}) - C(\vect u) \vert }{n^{-\omega \gamma}}\\
&=: I_{n1} + I_{n2} + I_{n3}.
\end{align*}

For $I_{n3}$, by Lipschitz-continuity of $C$ and by the choice of $\gamma$,
\begin{align*}
I_{n3} \le n^{1/2+\omega \gamma} \sqrt d \vert \vect u - \vect u^{(\vect \ell)}\vert =O( n^{1/2+\omega \gamma -\gamma})=o(1).
\end{align*}

For the treatment of $I_{n1}$, we can proceed similar as in \eqref{eq:empcop2} to obtain that $\vert G_n\{\vect G_n^- (\vect u)\}- G_n \{ \vect G_n^- (\vect u)^{(\vect \ell)} \}\vert \leq (d-2)n^{-\gamma} + o_P(n^{-1/2-\mu})$ for any $\mu \in (\omega , \theta_1)$. This yields $I_{n1}= o_P(n^{1/2+\omega \gamma - \gamma} + n^{\omega \gamma -\mu})=o_P(1)$.

Finally, regarding $I_{n2}$, note that $g_\omega(\vect u) = g_\omega(\vect u^{(\vect \ell)})$. Therefore,  
\[
I_{n2} = \sup_{\vect u \in S_\ell}  \frac{\vert \Cb_n(\vect u^{(\vect \ell)}) \vert }{g_\omega(\vect u^{(\vect \ell)})} 
\]
All  coordinates of vectors in $S_{\vect \ell}$ which are not equal to $1$ lie in $(n^{-\gamma}, 1-n^{-\gamma})$. Therefore, $I_{n2}$ can be treated similar as $R_n(M_{n\gamma})$.
\end{proof}

\begin{proof}[Proof of Lemma~\ref{lem:cbound2}.] Again, by a monotonicity argument, it suffices to treat sequences $\delta_n$ such that $\delta_n \gg n^{-1/2}$, i.e., $\delta_n \sqrt n \to \infty$.
Analogously to the proof of Lemma~\ref{lem:cbound} we can decompose the supremum. For $1/\{ 2 (1-\omega)\}<\gamma<1$ we can write
\[ 
\sup_{\vect u \in N(0,\delta_n)}\Big \vert \frac{\bar{\Cb}_n(\vect u)}{g_\omega(\vect u)} \Big \vert 
=
\sup_{\vect u \in N(0, n^{-\gamma})} \Big \vert  \frac{\bar{\Cb}_n (\vect u )}{g_\omega (\vect u)} \Big  \vert
\vee  
\sup_{\vect u \in N(n^{- \gamma}, \delta_n)} \Big \vert  \frac{\bar{\Cb}_n (\vect u )}{g_\omega (\vect u)} \Big  \vert 
=:
\bar{R}_{1n\gamma}\vee \bar{R}_{2n\gamma}.
\]
Therefore, we only have to show that $\bar R_{1n \gamma}=o_P(1)$ and $\bar R_{2n\gamma}=o_P(1)$. For both of these terms, we will distinguish the cases that either $g_\omega(\vect u) =u_1^\omega$ or $g_\omega(\vect u)=(1-u_1)^\omega$.
The cases  $g_\omega(\vect u) =u_\dimi^\omega$ or $g_\omega(\vect u)=(1-u_\dimi)^\omega$ for $\dimi >1$ can be treated similarly.

\medskip
\noindent
\textit{Treatment of $\bar R_{1n\gamma}$.}
First of all note that by definition 
\begin{align} \label{eq:barcnbound}
\bar{\Cb}_n(u_1, \dots , u_{\dimi-1}, 0 ,u_{\dimi+1} , \dots ,u_d)=\bar{\Cb}_n(\vect u^{(\dimi)})=0
\end{align}
and that $\alpha_n(u^{(\dimi)}) = \sqrt{n} \{ G_{n\dimi})(u_\dimi)-u_\dimi\}$
 for any $\dimi=1, \dots , d$. Let us first consider the supremum over those $\vect u \in N(0, n^{-\gamma})$ that additionally satisfy $g_\omega(\vect u)=u_1^\omega$. Choose $\omega' \in (\omega, \theta_2)$, then
\begin{align} \label{eq:barcng}
\Big \vert \frac{\bar{\Cb}_n(\vect u)}{g_\omega(\vect u)} \Big \vert  
&= \Big \vert \frac{\bar{\Cb}_n(\vect u) - \bar{\Cb}_n(0 , u_2, \dots , u_d)}{u_1^\omega} \Big \vert \\
&\leq  \sqrt n \Big \vert \frac{G_{n1}(u_1)}{u_1^\omega} \Big \vert +\sqrt n u_1^{1-\omega} + \sum_{\dimi=1}^d \Big \vert \frac{\dot C_\dimi (\vect u) \sqrt n \{ G_{n\dimi}(u_\dimi)- u_\dimi \}}{u_1^\omega} \Big \vert \nonumber \\
&\leq \Big \vert \frac{\sqrt n \{ G_{n1}(u_1 ) -u_1 \}}{u_1^{\omega'}} \Big \vert u_1^{\omega'-\omega} 
+2 \sqrt n u_1^{1-\omega}  \nonumber \\
& \hspace{5cm}+ \sum_{\dimi=1}^d \Big \vert \frac{\dot C_\dimi (\vect u) \sqrt n \{ G_{n\dimi}(u_\dimi)- u_\dimi \}}{u_1^\omega} \Big \vert \nonumber
\end{align}
By Condition \ref{cond:quantile}, the first summand on the right-hand side is of order $n^{-\gamma(\omega'-\omega)}O_P(1)=o_P(1)$, by the choice of $\omega'$. The second summand can be bounded by $2n^{1/2-\gamma(1-\omega)}=o(1)$, by the choice of $\gamma$. Thus, it remains to be shown that, for any $j=1, \dots, d$,
\begin{align}\label{eq:partials}
\bar S_{n\dimi} (\vect u)= \Big \vert \frac{\dot C_\dimi (\vect u) \sqrt n \{ G_{n\dimi}(u_\dimi)- u_\dimi \}}{u_1^\omega} \Big \vert = o_P(1)
\end{align}
uniformly in $\{\vect u \in N(0 ,n^{-\gamma}):g_\omega(\vect u)=u_1^\omega\}$. For later reference, we even show uniform negligibility on $\vect u \in N(0 ,\delta_n)$ such that $g_\omega(\vect u)=u_1^\omega$. $\bar S_{n1}$ can be bounded by the first term on the right-hand side of \eqref{eq:barcng}, and, therefore, is $O_P(\delta_n^{\omega' - \omega})$. Now, fix $j\in\{2, \dots, d\}$.
Uniformly in $\vect u \in N(0 ,\delta_n)$ such that $1-u_\dimi\leq u_1$, we have 
\[
\bar S_{n\dimi}(\vect u) 
\leq 
\sup_{u_\dimi \in (0,1)} \left \vert \frac{\sqrt n \{ G_{n \dimi}(u_\dimi)-u_\dimi\}}{(1-u_\dimi)^{\omega'}}\right \vert \times n^{-\gamma(\omega'-\omega)} =o_P(1)
\]
from Condition \ref{cond:quantile} and since $\vert \dot C_j (\vect u) \vert \leq 1$ for any $\vect u \in [0,1]^d$. Note that $\bar S_{n\dimi}(\vect u) =0$, if $u_j=1$. For the remaining case, i.e., for $\vect u$ such that $1-u_\dimi >u_1$, we may use the fact that $\dot C_\dimi(0,u_2, \dots ,u_d)=0$. A suitable application of the mean value theorem together with Condition \ref{cond:2nd} 
implies that, for any $\omega' \in (\omega, \theta_2)$,
\begin{align*}
\bar S_{n\dimi}(\vect u) 
&= \vert \ddot C_{1 \dimi}(\xi,u_2, \dots, u_d)  u_1 \sqrt n \{ G_{n \dimi}(u_\dimi) - u_\dimi \} \vert / {u_1^\omega}\\
&\leq u_1^{1-\omega}\{u_\dimi(1-u_\dimi)\}^{\omega'-1}K\frac{\vert \sqrt{n}\{G_{n \dimi}(u_\dimi) - u_\dimi \}\vert }{\{ u_\dimi(1-u_\dimi)\}^{\omega'}}\\
&\leq u_1^{\omega'-\omega}O_P(1)=o_P(\delta_n^{\omega'-\omega})=o_P(1).
\end{align*}

In order to finalize the treatment of $\bar R_{1n\gamma}$, let us 
now consider the case that $g_\omega (\vect u) =(1-u_1)^\omega$, i.e., $1-u_1=1-u_1 \wedge \dots \wedge \widehat{u_k} \wedge \dots \wedge u_d$ for some $ k \in \{ 2, \dots, d \}$. 
Without loss of generality, we may assume that $k=2$, which implies that $1-u_1 \le 1-u_2$ and $1-u_1 \ge 1-u_j$ for all $j \ge 3$. By definition of $\bar {\Cb}_n$ and by \eqref{eq:barcnbound}, we can write (note that $\bar{\Cb}_n(\vect u^{(2)} ) \equiv 0$ a.s.)
\begin{align}\label{eq:decomp}
\frac{\bar{\Cb}_n(\vect u)}{g_\omega(\vect u) } = \frac{\bar{\Cb}_n(\vect u ) - \bar{\Cb}_n(\vect u^{(2)})}{(1-u_1)^\omega}=\sum_{p=1}^4 \bar T_{np}(\vect u), 
\end{align}
where
\begin{align*}
\bar T_{n1}(\vect u) 
&= 
\frac{\sqrt n\{ G_n(\vect u) - G_n(\vect u^{(2)}) \}}{(1-u_1)^\omega} ,
\qquad \bar T_{n2}(\vect u) 
= 
\frac{\sqrt n \{ C(\vect u^{(2)}) - C(\vect u) \}}{(1- u_1)^\omega}, \\
\bar T_{n3}(\vect u ) 
&= 
- \sum_{\dimi \neq 2} \frac{\dot C_\dimi (\vect u )\sqrt n \{ G_{n \dimi}(u_\dimi) - u_\dimi \}}{(1-  u_1)^\omega}, \\
\bar T_{n4}(\vect u) 
&= 
\frac{\sqrt n \{ \dot C_2(\vect u^{(2)})- \dot C_2(\vect u) \} \{ G_{n2} (u_2) - u_2 \}}{(1-u_1)^\omega}.
\end{align*}
By Lipschitz-continuity of the copula function, we immediately obtain that 
\begin{align*}
\vert \bar T_{n2}(\vect u) \vert 
&\leq 
\sqrt n \sum_{\dimi \neq 2} (1- u_\dimi)(1-u_1)^{-\omega} \\
&\leq 
(d-1)\sqrt n (1- u_1)^{1-\omega} = O(n^{1/2-\gamma(1- \omega)})=o(1).
\end{align*}
For the estimation of $\bar T_{n1}$ we can proceed similarly as in \eqref{eq:empcop2} and obtain 
\begin{align*}
\vert \bar T_{n1}(\vect u) \vert 
&\leq
\frac{\sqrt n \sum_{ \dimi \neq 2} \vert 1- G_{n \dimi}(u_\dimi) \vert }{(1- u_1)^\omega} \\
&\leq 
\sum_{\dimi \neq 2 }\frac{\vert \sqrt n \{ G_{n \dimi }(u_\dimi) - u_\dimi \} \vert}{(1- u_\dimi)^{\omega'}} (1-u_\dimi)^{\omega'-\omega} + (d-1)\sqrt n (1- u_1)^{1-\omega} \\
&=
 O_P(n^{-\gamma (\omega' - \omega)}) + O(n^{1/2 - \gamma(1-\omega)}) =o_P(1)
\end{align*}
uniformly in $\{\vect u \in N(0 ,n^{-\gamma})\cap (0,1)^d:g_\omega(\vect u)=(1-u_1)^\omega\}$, by the choice of $\gamma$ and by choosing $\omega'\in(\omega, \theta_2)$. Note that the terms with $u_j=1$ vanish immediately from the definition of $G_{nj}$.
Similarly, 
using the fact that $\vert \dot C_\dimi(\vect u) \vert \leq 1 $ for all $\dimi=1 , \dots ,d $ and for all $\vect u \in [0,1]^d$, we get that
\begin{align*}
\vert \bar T_{n3} (\vect u ) \vert 
\leq 
\sum_{\dimi \neq 2}\sup_{u_\dimi \in (0,1)}  \Big \vert \frac{\sqrt n \{ G_{n \dimi}(u_\dimi) - u_\dimi \}}{(1-u_\dimi)^\omega} \Big \vert 
= 
O_P(n^{-\gamma(\omega' - \omega)}) = o_P(1)
\end{align*} 
uniformly in $\{\vect u \in N(0 ,n^{-\gamma}):g_\omega(\vect u)=(1-u_1)^\omega\}$,  by choosing $\omega'\in(\omega, \theta_2)$. Finally, in order to bound the remaining term $\bar T_{4n}$, we may use the mean value theorem and  Conditions \ref{cond:2nd} and \ref{cond:quantile} to obtain that
\begin{align*}
\vert \bar T_{n4 }(\vect u) \vert 
& =
\frac{\sqrt n \sum_{ \dimi \neq 2} \vert  \ddot C_{2 \dimi}( \vect \xi) (1- u_\dimi) \{ G_{n 2}(u_2) - u_2 \} \vert }{(1-u_1)^\omega}\\
& \leq 
\sum_{\dimi \neq 2} O_P(1) (1- u_1)^{1- \omega} \{u_2(1- u_2) \}^{\omega' - 1} 
=
O_P(n^{-\gamma(\omega' - \omega)})=o_P(1)
\end{align*}
uniformly in $\{\vect u \in N(0 ,n^{-\gamma}):g_\omega(\vect u)=(1-u_1)^\omega\}$,  with some intermediate value $\vect \xi=(\xi _1,u_2, \xi_3, \dots \dots ,\xi_d)' \in (0,1)^d$ between $\vect u^{(2)}$ and $\vect u$ and where the last estimation follows by choosing $\omega'$ in $(\omega , \theta_2)$.

\medskip
\noindent
\textit{Treatment of $\bar R_{2n \gamma}$.} First suppose that $g_\omega(\vect u)=u_1^\omega$. Then we write
\begin{align*}
\Big \vert \frac{\bar {\Cb}_n(\vect u) }{u_1^\omega} \Big \vert 
&= 
\Big \vert \frac{\bar {\Cb}_n(\vect u) - \bar {\Cb}_n(0,u_2, \dots , u_d)}{u_1^\omega} \Big \vert \\
& \leq
 \frac{\vert \alpha_n(\vect u) -\alpha_n(0 , u_2, \dots , u_d) \vert}{u_1^\omega}+\sum_{\dimi =1}^d \bar S_{n\dimi}(\vect u),
\end{align*}
where $\bar S_{n \dimi}(\vect u)$ is defined in \eqref{eq:partials}. Negligibility of $\bar S_{n\dimi}$ in the latter decomposition has been shown subsequent to \eqref{eq:partials}. From Condition \ref{cond:modulus}, the first term on the right-hand side of the last display can be bounded by $(\{u_1^{\mu-\omega}\vee n^{- \mu }\}u_1^{-\omega})o_P(1)=o_P(\delta_n^{\mu - \omega} \vee n^{-\mu + \gamma \omega})$, which vanishes as $n \to \infty$ if we choose $\mu \in (\omega , \theta_1)$.

Now, suppose $g_\omega(\vect u) = (1-u_1)^\omega$, i.e., $1-u_1=1-u_1 \wedge \dots \wedge \widehat{u_k} \wedge \dots \wedge u_d$ for some $ k \in \{ 2, \dots, d \}$. 
Again, without loss of generality, we may assume that $k=2$, which implies that $1-u_1 \le 1-u_2$ and $1-u_1 \ge 1-u_j$ for all $j \ge 3$. 
We decompose
\begin{align*}
\Big \vert \frac{\bar {\Cb}_n(\vect u)}{g_\omega(\vect u)} \Big \vert = \Big \vert  \frac{\bar {\Cb}_n(\vect u) - \bar {\Cb}_n(\vect u^{(2)})}{g_\omega(\vect u)} \Big \vert \leq \frac{\vert \alpha_n(\vect u) - \alpha_n(\vect u^{(2)}) \vert }{(1-u_1)^\omega}+\vert \bar T_{n3}(\vect u) \vert + \vert \bar T_{n4}(\vect u)\vert,
\end{align*}
where $\bar T_{n3}(\vect  u) $ and $\bar T_{n4}(\vect u)$ are defined in \eqref{eq:decomp}. By the same arguments as for their treatment  on $N(0, n^{-\gamma})$, we have $\vert \bar T_{n3}(\vect u)\vert =o_P(1)$ and $\vert \bar T_{n4}(\vect u)\vert =o_P(1)$, uniformly in $\vect u \in N(0, \delta_n)$ with $g_\omega(\vect u ) =(1-u_1)^\omega$. The remaining term on the right-hand side of the last display can be bounded by an application of Condition~\ref{cond:modulus}. Choosing $\mu \in (\omega ,\theta_1)$, we obtain
\begin{align*}
\frac{\vert \alpha_n(\vect u) - \alpha_n(\vect u^{(2)}) \vert }{(1-u_1)^\omega} \leq \sum_{\dimi \neq 2} \frac{(1- u_\dimi)^\mu  \vee n^{-\mu}}{(1-u_1)^\omega}o_P(1)=o_P(\delta_n^{\mu - \omega} \vee n^{- \mu + \omega\gamma}),
\end{align*}
which is $o_P(1)$.
This completes the proof.
\end{proof}

\begin{proof}[Proof of Lemma~\ref{lem:contc}]
Let us first show \eqref{eq:cbtight}.
As in the proof of Lemma~\ref{lem:cbound}, by a monotonicity argument, it suffices to treat sequences $\delta_n$ such that $\delta_n \ge n^{-1/2}$. 
We split the proof into two cases and begin by considering $\vect u \in N(\bbb n^{-1},2\delta_n^{1/2})$.  
Obviously, $\vert \vect u - \vect u' \vert \leq \delta_n$ implies $\vect u' \in N(\bbb n^{-1}, 2\delta_n^{1/2}+\delta_n) \subset N(\bbb n^{-1}, 3 \delta_n^{1/2})$.  Thus, by Lemma~\ref{lem:cbound}, we obtain 
\begin{align*}
\sup_{\vect u, \vect u' \in [\frac{c}{n}, 1- \frac{c}{n}]^d, \vert \vect u - \vect u' \vert  \leq \delta_n, \vect u \in N(\bbb n^{-1},2\delta_n^{1/2})} \Big \vert \frac{\Cb_n( \vect u )}{ g_\omega(\vect u)} - \frac{\Cb_n(\vect u')}{g_\omega (\vect u') } \Big \vert =o_P(1).
\end{align*}
Now, consider the case $\vect u \in  N(2 \delta_n^{1/2}, 1/2)$. Then, $\vert \vect u -\vect u' \vert \leq \delta_n$ implies that 
$\vect u' \in N(2 \delta_n^{1/2}-\delta_n , 1/2)\subset  N(\delta_n^{1/2}, 1/2)$. 
Hence, Lemma~\ref{lem:epsdiff} implies that
\begin{multline*}
 \sup_{\vect u,  \vect u' \in [\frac{\bbb}{n}, 1- \frac{\bbb}{n}]^d, \vert \vect u - \vect u' \vert  \leq \delta_n, \vect u \in  N(2 \delta_n^{1/2}, 1/2)} \Big \vert \frac{\Cb_n( \vect u )}{ g_\omega(\vect u)} - \frac{\Cb_n(\vect u')}{g_\omega (\vect u') } \Big \vert   \\
\le \sup_{\vect u, \vect u' \in[\frac{\bbb}{n}, 1- \frac{\bbb}{n}]^d \cap N(\delta_n^{1/2}, 1/2) , \vert \vect u - \vect u' \vert \leq \delta_n} \Big \vert \frac{ \bar \Cb_n( \vect u )}{ g_\omega(\vect u)} - \frac{\bar \Cb_n(\vect u')}{g_\omega (\vect u') } \Big \vert  + o_P(1).
\end{multline*}
Therefore, in order to prove \eqref{eq:cbtight}, it suffices to show that
\begin{align}
&   \sup_{\vect u, \vect u' \in N(\delta_n^{1/2}, 1/2) , \vert \vect u - \vect u' \vert \leq \delta_n} \Big \vert \frac{\bar \Cb_n(\vect u ) - \bar \Cb_n(\vect u')}{g_\omega(\vect u)} \Big \vert =o_P(1) 
\label{eq:eq1}\\
 & \sup_{\vect u, \vect u' \in N(\delta_n^{1/2}, 1/2) , \vert \vect u - \vect u' \vert \leq \delta_n} \Big \vert \bar \Cb_n(\vect u') \Big ( \frac{1}{g_\omega (\vect u)} - \frac{1}{g_\omega(\vect u') } \Big ) \Big \vert   =o_P(1).
\label{eq:eq2}
\end{align}
The respective proofs will be given below at the end of this proof.

For the proof of \eqref{eq:barcbtight}, note that $\bar \Cb_n(\vect u) / \tilde g_\omega(\vect u) = 0$ for $g_\omega(\vect u) = 0$. Therefore, we can bound
\begin{align*}
\sup_{\vert \vect u - \vect u' \vert  \leq \delta_n, g_1(\vect u)=0, g_1(\vect u')>0} \Big \vert \frac{\bar \Cb_n( \vect u )}{ \tilde g_\omega(\vect u)} - \frac{\bar \Cb_n(\vect u')}{\tilde g_\omega (\vect u') } \Big \vert 
\le 
\sup_{\vect u' \in N(0,\delta_n)}\Big \vert  \frac{\bar \Cb_n(\vect u')}{g_\omega (\vect u') } \Big \vert
=o_P(1)
\end{align*}
by Lemma~\ref{lem:cbound2}.
The suprema over $\{ \vect u: g_1(\vect u)>0, g_1(\vect u')=0\}$ or  $\{ \vect u: g_1(\vect u)= g_1(\vect u')=0 \}$ can be treated analogously, whereas the suprema over $\{ \vect u: g_1(\vect u)>0, g_1(\vect u')>0\}$ can be handled by \eqref{eq:eq1}, \eqref{eq:eq2} and Lemma \ref{lem:cbound2}. This proves \eqref{eq:barcbtight}.

It remains to be shown that \eqref{eq:eq1} and \eqref{eq:eq2} are valid.

\medskip

\noindent
\textit{Proof of \eqref{eq:eq1}.} By Condition \ref{cond:2nd} and \ref{cond:modulus} and the fact that $\dot C_j \in [0,1]$ we have, for $\vect u , \vect u' \in  N(\delta_n^{1/2}, 1/2)$, $\vert \vect u - \vect u' \vert \leq \delta_n$ and any $\mu\in(0,\theta_1)$,
\begin{align*}
\left \vert \frac{\bar \Cb_n(\vect u ) - \bar \Cb_n(\vect u')}{g_\omega (\vect u)} \right \vert 
&\leq  
\left \vert \frac{\alpha_n(\vect u ) - \alpha_n (\vect u')}{g_\omega (\vect u )} \right \vert 
+ \sum_{\dimi=1}^d \left \vert \frac{\dot C_{\dimi}(\vect u) \{ \alpha_n(\vect u^{(\dimi)} ) - \alpha_n (\vect u'^{(\dimi)}) \}}{g_\omega (\vect u )} \right \vert \\
& \hspace{2cm}
+ \sum_{\dimi =1}^d \left \vert \frac{\{ \dot C_\dimi(\vect u ) - \dot C_\dimi(\vect u') \} \alpha_n(\vect u'^{(\dimi)})}{g_\omega (\vect u)} \right \vert .
\end{align*}
The right-hand side can be further bounded by
\begin{align*}
&  (d+1) \frac{\vert \vect u - \vect u' \vert^\mu \vee n^{-\mu} }{g_\omega (\vect u)}  M_n(\delta_n, \mu)
+\sum_{\dimi =1}^d \left \vert \frac{\{ \dot C_\dimi(\vect u ) - \dot C_\dimi(\vect u') \} \alpha_n(\vect u'^{(\dimi)})}{g_\omega (\vect u)} \right \vert .
\end{align*}
Since $g_\omega(\vect u)\geq \delta_n^{\omega/2}$ for $\vect u \in  N( \delta_n^{1/2}, 1/2)$ , the first summand on the right of the last display is of order $O_P(\delta_n^{\mu-\omega/2})$, which is $o_P(1)$ if we choose $\mu > \omega/2$. For the second term, we fix $j$ and will consider two cases for each summand separately. First, suppose $1-u_j' < \delta_n^{1/2}$. In this case, Condition~\ref{cond:quantile} yields, for arbitrary $\omega' \in (0,\theta_2)$,
\begin{align*}
 \left \vert \frac{\{ \dot C_\dimi(\vect u ) - \dot C_\dimi(\vect u') \} \alpha_n(\vect u'^{(\dimi)})}{g_\omega (\vect u)} \right \vert 
&\leq 
2\delta_n^{-\omega/2}\{ u'_j (1- u'_j) \}^{\omega'}\frac{\vert \alpha_n(\vect {u'}^{(j)})\vert }{\{ u'_j (1- u'_j) \}^{\omega'}} \\
&=O_P(\delta_n^{-\omega/2 + \omega'/2}).
\end{align*}
Since we can choose $\omega' \in (\omega , \theta_2)$ the latter is $o_P(1)$.

Now, suppose $1-u_j'  \geq \delta_n^{1/2}$. Then, the mean value theorem allows to write
\[
 \left \vert \frac{\{ \dot C_\dimi(\vect u ) - \dot C_\dimi(\vect u') \} \alpha_n(\vect u'^{(\dimi)})}{g_\omega (\vect u)} \right \vert 
\leq
 \sum_{ \ell =1}^d \left \vert \frac{\ddot C_ {\dimi \ell } (\vect \xi_{\dimi}) \alpha_n(\vect u'^{(\dimi)}) (u_{\ell} - u_{\ell}')}{g_\omega (\vect u)} \right \vert , 
\]
where $\vect \xi_j$ denotes an intermediate point between $\vect u$ and $\vect u'$. In particular, the components of $\vect \xi_j=(\xi_{j1}, \dots, \xi_{jd})$ satisfy  $\xi_{\dimi \ell} \ge \sqrt \delta_n$ and $1- \xi_{\dimi\dimi} \ge \sqrt \delta_n - \delta_n \ge \sqrt \delta_n / 2$, for sufficiently large $n$.  Then, by Condition~\ref{cond:2nd}, the sum on the right-hand side of the last display can be bounded by
\begin{align*}
d  \frac{K}{\xi_{\dimi\dimi}(1-\xi_{\dimi\dimi})}\vert \alpha_n(\vect u'^{(\dimi)}) \vert \delta_n^{1-\omega/2} 
=  O_P(\delta_n^{1/2-\omega/2} ) =o_P(1).
\end{align*}

\medskip

\noindent
\textit{Proof of \eqref{eq:eq2}.} Note that it is sufficient to bound $\vert g_\omega(\vect u)^{-1} - g_\omega(\vect u' )^{-1} \vert$, because $\sup_{\vect u \in [0,1]^d} \vert \bar \Cb_n(\vect u) \vert =O_P(1)$. To this end, we first observe that, for $\vect u, \vect u' \in N(\delta_n^{1/2}, 1/2)$ and $\vert \vect u - \vect u' \vert \leq \delta_n$, we have
\[
\vert g_\omega(\vect u) - g_\omega(\vect u') \vert
\le
\omega \delta_n^{(\omega-1)/2} \vert g_1(\vect u) - g_1(\vect u') \vert
=
O(\delta_n^{(\omega+1)/2})
\]
where we used  the mean value theorem and the fact that $g_1$ is Lipschitz-continuous on $N( \delta_n^{1/2}, 1/2)$. Therefore,
\begin{align*}
\Big \vert \frac{1}{g_\omega(\vect u)} -\frac{1}{g_\omega(\vect u')} \Big \vert 
= 
\Big \vert \frac{g_\omega( \vect u') - g_\omega (\vect u)}{g_\omega(\vect u) g_\omega(\vect u')}\Big \vert 
= 
O (\delta_n^{(\omega+1)/2 - \omega}) =o(1),
\end{align*}
which implies \eqref{eq:eq2}.
\end{proof}

\subsection{Proof of Theorem~\ref{theo:score}.} \label{subsec:proofs2}

Let $n\ge 2$. Decompose $\sqrt n \{ R_n - \mathbb{E} [ J(\vect U) ] \} =A_n - r_{n1}$, where
\begin{align*}
A_n 
&= 
\sqrt n \int_{(\frac{1}{2n},1-\frac{1}{2n}]^2} J(\vect u) \mathrm{d}(\hat C_n - C)(\vect u) \\
r_{n1} 
&= 
\sqrt n \int_{\{(\frac{1}{2n},1-\frac{1}{2n}]^2\}^c}J(\vect u) \mathrm{d}C(\vect u),
\end{align*}
where $A^c$ denotes the complement of a set $A$ in $(0,1)^2$. From integration by parts for Lebesgue-Stieltjes integrals (see Theorem \ref{theo:partialgen} in the supplementary material) we have that $A_n= B_{n}+r_{n2}+r_{n3}$, where
\[
B_n =  \int_{(\frac{1}{2n},1-\frac{1}{2n}]^2} \hat{\mathbb{C}}_n(\vect u) \mathrm{d} J(\vect u) 
\]
where 
\begin{multline*}
r_{n2} 
= 
\Delta(\hat{\mathbb{C}}_n J,   \tfrac{1}{2n},\tfrac{1}{2n},1-\tfrac{1}{2n},1-\tfrac{1}{2n}) \\
-  \int_{(\frac{1}{2n},1-\frac{1}{2n}]} \hat{\mathbb{C}}_n( u, 1-\tfrac{1}{2n} )J(\mathrm{d} u, 1-\tfrac{1}{2n})
+ \int_{(\frac{1}{2n},1-\frac{1}{2n}]} \hat{\mathbb{C}}_n(u,\tfrac{1}{2n}) J(\mathrm{d} u, \tfrac{1}{2n})  \\
- \int_{(\frac{1}{2n},1-\frac{1}{2n}]} \hat{\mathbb{C}}_n( 1- \tfrac{1}{2n},v )  J(1-\tfrac{1}{2n}, \mathrm{d} v) 
+  \int_{(\frac{1}{2n},1-\frac{1}{2n}]}\hat{\mathbb{C}}_n(\tfrac{1}{2n},v)  J(\tfrac{1}{2n},\mathrm{d}  v),
\end{multline*}
with $\Delta(f,a_1,a_2,b_1,b_2) = f(b_1,b_2)-f(a_1,b_2)-f(b_1,a_2) +f(a_1,a_2)$ for $f:(0,1)^2\to \R$ and $\vect a, \vect b \in (0,1)^2$ and where
\begin{align*}
r_{n3} 
&=
 \int_{(\frac{1}{2n},1-\frac{1}{2n}]^2} \nu_n(\{u\}\times (v,1-\tfrac{1}{2n}]) + \nu_n((u,1-\tfrac{1}{2n}] \times\{v\}) \\ 
& \hspace{7.3cm}
+ \nu_n(\{(u,v)\}) \mathrm{d}J(u,v)\\
&\hspace{1.8cm}+
\int_{(\frac{1}{2n},1-\frac{1}{2n}]} \nu_n(\{u\}\times (\tfrac{1}{2n},1-\tfrac{1}{2n}]) J(\mathrm{d}u, \tfrac{1}{2n}) \\
&\hspace{3.6cm}+ 
\int_{(\tfrac{1}{2n},1-\tfrac{1}{2n}]} \nu_n( (\tfrac{1}{2n},1-\tfrac{1}{2n}]\times\{v\}) J(\tfrac{1}{2n},\mathrm{d}v) ,
\end{align*}
with $\nu_n$ denoting the unique signed measure on $[\tfrac{1}{2n},1-\tfrac{1}{2n}]$ associated with~$\hat \Cb_n$ (see Theorem~\ref{theo:decomp2} in the supplementary material).

For the arguments that follow, we remark that by Proposition \ref{prop:wdcheckmod} the conditions of Theorem \ref{theo:weightalpha} imply those of Theorem \ref{theo:weight}. Thus, all results from the proof of Theorem \ref{theo:weight} are applicable here.
 
Regarding weak convergence of $B_n$, observe that by Theorem~\ref{theo:weightalpha}, Lem\-ma~\ref{lem:cbound2} and the integrability condition in \eqref{cond:score}
\begin{align*}
B_n  & =  \int_{(0,1)^2} \ind \{ \vect u \in(\tfrac{1}{2n},1-\tfrac{1}{2n}]^2 \} \frac{\bar{\mathbb{C}}_n(\vect u)}{g_\omega(\vect u)}  g_\omega(\vect u) \, \mathrm{d} J(\vect u) + o_P(1)
\\
& = \int_{(0,1)^2}  \frac{\bar{\mathbb{C}}_n(\vect u)}{ g_\omega(\vect u)} g_\omega(\vect u)\, \mathrm{d} J(\vect u) + o_P(1).
\end{align*}
Now, the integrability condition in \eqref{cond:score} implies that the functional $f \mapsto \int_{(0,1)^2} f \tilde g_\omega\, \mathrm{d} J$ is continuous when viewed as a map from $(\ell^\infty((0,1)^2),\|\cdot\|_\infty)$ to $\R$, and thus $B_n$ converges weakly to $\int_{(0,1)^2  }\mathbb{C}_C (\vect u)\mathrm{d}  J(\vect u)$ by Theorem \ref{theo:weightalpha} and the continuous mapping theorem.
Hence, it remains to be shown that $r_{n1}$, $r_{n2}$ and $r_{n3}$ are $o_P(1)$.

Regarding $r_{n1}$, 
since $\vert J (u,v ) \vert \leq \const \times g_\omega (u,v)^{-1}$,  we can bound
\[
\vert r_{n1} \vert \leq \sqrt n  \int_{([\frac{1}{2n},1-\frac{1}{2n}]^2 )^c} g_{\omega}(u,v )^{-1} \mathrm{d}C (u,v).
\]
The set $\{( \frac{1}{2n}, 1-\frac{1}{2n}]^2 \}^c$ consists of vectors where either both components or only one component is close to the boundary of $[0,1]^2$. In order to bound the integral on the right-hand side of the last display, we distinguish these cases and exemplarily consider the integral over $(0,\frac{1}{2n}]^2$ and the one over $(0,\frac{1}{2n}] \times (\frac{1}{2n},1-\frac{1}{2n}]$. Integrals over the remaining subsets can be treated in the same way. First, since $g_\omega(u,v)^{-1} \le u^{-\omega}+v^{-\omega}$ for $u,v \in (0,\frac{1}{2n}]$, we have
\begin{align*}
\sqrt n  \int_{( 0,\frac{1}{2n}]^2}g_{\omega}(u,v )^{-1}\,  \mathrm{d}C (u,v) 
\leq 
\sqrt n \int_{( 0, \frac{1}{2n}]^2} u^{-\omega}+v^{-\omega} \mathrm{d}C(u,v).
\end{align*}
Let us only consider the integral over $u^{-\omega}$ on the right-hand side, the one over $v^{-\omega}$ can be treated analogously. We have
\begin{align*}
\sqrt n \int_{( 0, \frac{1}{2n} ]^2} u^{-\omega} \mathrm{d}C(u,v)
&\leq 
\sqrt n \int_{( 0, \frac{1}{2n}] \times [0,1]} u^{-\omega} \, \mathrm{d} C(u,v) \\
&= \sqrt n \int_{( 0, \frac{1}{2n}]} u^{-\omega} \, \mathrm{d} u
  = O(n^{-1/2+\omega}) = o(1).
\end{align*}
Second, on $(0,\frac{1}{2n}] \times (\frac{1}{2n},1-\frac{1}{2n}]$, we have $g_\omega(u,v)^{-1} = u^{-\omega}$ , whence, by a similar reasoning, 
\begin{align*}
\sqrt n \int_{(0,\frac{1}{2n}] \times (\frac{1}{2n},1-\frac{1}{2n}]} g_\omega(u,v)^{-1} \mathrm{d}C(u,v) 
\leq  \sqrt n \int_{( 0, \frac{1}{2n}]}u^{-\omega}\,\mathrm{d} u = O(n^{-1/2+\omega}).
\end{align*}

Regarding $r_{n2}$,
use Theorem \ref{theo:weightalpha} and \eqref{cond:score2} and \eqref{cond:score3} to replace $\hat \Cb_n/g_\omega$ by $\bar \Cb_n/g_\omega$ at the cost of a negligible remainder (note that $g_\omega(u,\delta) = \delta^\omega$ for $u\in(\delta, 1-\delta]$) . Then, the four integrals in the definition of $r_{n2}$ are $o_P(1)$ by \eqref{cond:score2}, \eqref{cond:score3}, Lemma~\ref{lem:cbound2} and Proposition~\ref{prop:wdcheckmod}, while $\Delta(\bar \Cb_n J,   \tfrac{1}{2n},\tfrac{1}{2n},1-\tfrac{1}{2n},1-\tfrac{1}{2n})$ converges to $0$ by Lemma~\ref{lem:cbound2}, Proposition~\ref{prop:wdcheckmod} and the fact that $\vert J(\vect u)\vert  \leq \const \times g_\omega(\vect u)^{-1}$ for $\vect u \in (0,1)^2$.

Regarding $r_{n3}$, since $\hat C_n$ and $C$ are completely monotone, the (unique) measures in the Jordan decomposition of $\nu_n$ are given by $\nu_n^{+}=\sqrt n \nu_{\hat C_n}$ and $\nu_n^{-}=\sqrt n  \nu_C$, where $\nu_{\hat C_n}$ and $\nu_{C}$ denote the measures corresponding to $\hat C_n$ and $C$, respectively. Thus, continuity of the copula $C$ yields 
\[
 \nu_{n}(\{u\}\times (v,1-\tfrac{1}{2n}]) = \sqrt n \nu_{\hat C_n}(\{u\}\times (v,1-\tfrac{1}{2n}]) \leq \sqrt n \{ \hat C_n(u,1)- \hat C_n(u-,1) \}.
\]
Since the last display is bounded by $ n^{-1/2}$ times the maximum number of $\hat U_{i1}$ that are equal, a reasoning which is similar to the one used to obtain \eqref{eq:maxjump} yields that, for any $\mu \in(\omega,1/2)$,
\[ 
\nu_{n}(\{u\}\times (v,1-\tfrac{1}{2n}]) = O_P(n^{-\mu}) 
\]
uniformly in $u,v \in (0,1)^2$. Similar estimations for the remaining terms in $r_{n3}$ imply that $\vert r_{n3}\vert $ is of the order
\begin{align*}
O_P(n^{-\mu}) \bigg  \{ \int_{(\frac{1}{2n},1-\frac{1}{2n}]^2} \vert \mathrm{d}J\vert  +\int_{(\frac{1}{2n},1-\frac{1}{2n}]} \vert  J(\mathrm{d}u, \tfrac{1}{2n}) \vert 
+ \int_{(\tfrac{1}{2n},1-\tfrac{1}{2n}]} \vert J(\tfrac{1}{2n},\mathrm{d}v)\vert \bigg \}.
\end{align*}
By Conditions \eqref{cond:score}--\eqref{cond:score3}, these integrals are of order $O(n^{\omega})$ which leads to $\vert r_{n3}\vert =O_P(n^{\omega-\mu})=o_P(1)$.
\qed

\section{Auxiliary results} \label{sec:aux}
\def\theequation{6.\arabic{equation}}
\setcounter{equation}{0}

\begin{lemma}\label{lem:inverse}
Suppose Condition~\ref{cond:modulus} is met. Then, for $\dimi=1, \dots , d$ and any $\mu \in [0,\theta_1)$, we have
\begin{align*}
\sup_{u \in [0,1]} \vert G_{n\dimi}\{G_{n\dimi}^-(u)\} - u \vert =o_P(n^{-1/2-\mu}).
\end{align*}

\end{lemma}

\begin{proof}
From the definition of the (left-continuous) generalized inverse, we have that $\sup_{u \in [0,1]}\vert H\{H^-(u) \} -u \vert$ is bounded by the maximum jump heigth of the function $H$, i.e., 
\begin{align*}
\sup_{u \in [0,1]} \vert  G_{n\dimi}\{G_{n\dimi}^-(u)\} - u \vert& \leq \sup_{u \in [0,1]} \vert G_{n\dimi}(u) - G_{n \dimi}(u-)\vert
\end{align*}
Therefore, the assertion follows from \eqref{eq:maxjump} and Condition \ref{cond:modulus}.
\end{proof}

\begin{lemma}\label{lem:uxi} Suppose Condition~\ref{cond:quantile} is met. Then, for $j=1, \dots, d$ and any $\gamma \in(0, \{1/ [2(1-\theta_2)]\}\wedge \theta_3)$, we have
\begin{align*} 
K_{nj}(\gamma)=\sup \left| \frac{u_j(1-u_j)}{\xi_j(1-\xi_j)} \right| = O_P(1),
\end{align*}
where the supremum is taken over all $u_j \in [n^{-\gamma}, 1-n^{- \gamma}]$ and all $\xi_j$ between $G_{nj}^-(u_j)$ and $u_j$.
\end{lemma}

\begin{proof}[Proof of Lemma~\ref{lem:uxi}.]
Since
\[
K_{nj}(\gamma) \le K_{nj}^{(1)}(\gamma) \times K_{nj}^{(2)}(\gamma) := \sup \left| \frac{u_j}{\xi_j} \right| \times \sup \left| \frac{1-u_j}{1-\xi_j} \right|,
\]
it suffices to treat both suprema on the right-hand side separately. In the following, we only consider the first one; the second one can be treated along similar lines. Obviously,
\[
K_{nj}^{(1)}(\gamma) \le 1 \vee \sup_{u_j \in [n^{-\gamma}, 1-n^{- \gamma}] } \frac{u_j}{G_{nj}^-(u_j)}
\]
Let $\Omega_n$ denote the event that $\sup_{u_j\in[n^{-\gamma}, 1- n^{-\gamma}] }\vert {\{ G_{nj}^-(u_j) - u_j \} }/{u_j} \vert \le 1/2$. Choose $\omega' \in (0\vee (1-\tfrac{1}{2\gamma}), \theta_2)$ and use Condition~\ref{cond:quantile} to conclude that
\begin{align*}
\sup_{u_j\in[n^{-\gamma}, 1-n^{-\gamma}]  }\left| \frac{ G_{nj}^-(u_j) - u_j  }{u_j} \right| 
& \le
\sup_{u_j\in[n^{-\gamma}, 1-n^{-\gamma}]  } \left\{ \sqrt n\left| \frac{ G_{nj}^-(u_j) - u_j  }{u_j^{\omega'}} \right| \times\frac{ u_j^{\omega'-1} }{\sqrt n} \right\} \\
& =
O_P(n^{-1/2-\gamma(\omega'-1)} ) = o_P(1).
\end{align*}
Thus, $\Prob(\Omega_n^c) = o(1)$, which implies 
\begin{align*}
\sup_{u_j \in [n^{-\gamma}, 1-n^{- \gamma}] }  \frac{u_j }{ G_{nj}^-(u_j)   }
&= 
\sup_{u_j \in [n^{-\gamma}, 1-n^{- \gamma}] }  \left( 1 +  \frac{G_{nj}^-(u_j) - u_j}{u_j} \right)^{-1} \ind_{\Omega_n} + o_P(1) \\
& \le
2+o_P(1) =O_P(1),
\end{align*}
where we used that $1/(1+x) \le 1/(1-|x|)$ for $x\in [-1/2, 1/2]$.
This yields the assertion.
\end{proof}

\begin{lemma}\label{lem:betabound}
Under Condition~\ref{cond:modulus}, Conidtion~\ref{cond:weak} and Condition~\ref{cond:quantile} we have for any $\omega \in (0,{\theta_1 \wedge}\theta_2)$ and any $\gamma > 1/2$
\[
\sup_{u_\dimi \in [1-n^{-\gamma}, 1]} |\beta_{n\dimi}(u_\dimi)| = o_P(n^{-\omega/2}).
\]
\end{lemma}

\begin{proof} Since the result is one-dimensional, we drop the index $\dimi$ in the following. Note that all the arguments that follow lead to bounds which are valid uniformly in $u  \in [1-n^{-\gamma}, 1]$. Now, fix $u  \in [1-n^{-\gamma}, 1]$ and  choose $i\in \{0, \dots, n-1\}$ such that $u \in (\frac{i}{n}, \frac{i+1}{n}]$. Then, $G_{n}^-(u) = U_{i+1:n}$, where $U_{1:n} \le \dots \le U_{n:n}$ denote the order statistics of $U_1, \dots, U_n$. Hence,
\begin{align*}
n^{\omega/2} | \beta_{n}(u)  |
& \le
n^{\omega/2 +1/2} \{ |U_{i+1:n} - i/n | \vee |U_{i+1:n} - (i+1)/n | \} \\
& \le 
n^{\omega/2 +1/2} |U_{i+1:n} - i/n |  + n^{-1/2+\omega/2}
\end{align*}
Now, as a consequence of Lemma~\ref{lem:inverse}, we have $G_n(U_{i+1:n}) = G_n\{ G_n^-(u) \} = i/n + \kappa_{i,n}$, where
$\max_{i=0}^{n-1} \kappa_{i,n} = o_P(n^{-\mu-1/2})$ with $\mu \in (\omega/2, \theta_1)$. Therefore,
\[
n^{\omega/2 +1/2} |U_{i+1:n} - i/n |  
\le 
n^{\omega/2 +1/2} |G_n(U_{i+1:n}) - U_{i+1:n} | +  n^{\omega/2 +1/2}  \kappa_{i,n}
\]
The second term on the right-hand side is $o_P(n^{-\mu+\omega/2})=o_P(1)$. For the first term, we have
\begin{multline*}
n^{\omega/2 +1/2} |G_n(U_{i+1:n}) - U_{i+1:n} | 
=
\frac{\alpha_n(U_{i+1:n}) } { (1-U_{i+1:n})^\omega } n^{\omega/2} (1-U_{i+1:n})^\omega \\
\le 
\sup_{u\in(0,1)} \frac{|\alpha_n(u)|}{(1-u)^\omega}  \times n^{\omega/2} (1-U_{i+1:n})^\omega
=O_P(1) \times \{ \sqrt n (1-U_{i+1:n}) \}^{\omega}
\end{multline*}
For the factor on the right, since $u\ge 1-n^{\gamma}$, we have, for any $w \in (\frac{i}{n}, \frac{i+1}{n}]$,
\begin{align*}
\sqrt n ( 1- U_{i+1:n}) 
&= \sqrt n \{ w - G_n^-(w) + 1 - w\}  \\
&\hspace{-1cm}\le 
\sup_{v \in[1-n^{-\gamma},1]} |\beta_n(v)| + n^{1/2-\gamma}
\\
&\hspace{-1cm}\le \sup_{v \in[1-n^{-\gamma},1]} |\beta_n(v) - \beta_n(1-n^{-1/2})| + |\beta_n(1-n^{-1/2})| + n^{1/2-\gamma}.
\end{align*}
The first term in the expression above is $o_P(1)$ by asymptotic equicontinuity of $\beta_n$ (which follows from weak convergence of $\beta_n$ to a Gaussian process, this is a consequence of Condition~\ref{cond:weak} and the functional delta method), the second term is $o_P(1)$ by Conidtion~\ref{cond:quantile}, and the third term vanishes since $\gamma>1/2$.
\end{proof}

%

\bibliographystyle{chicago}
\bibliography{biblio}

\newpage

\appendix
\begin{center}
{\LARGE Supplement to \\ ``Weak convergence of the empirical copula process with respect to weighted metrics''}

\vspace{.5cm}
{\large Betina Berghaus, Axel B\"ucher and Stanislav Volgushev

\vspace{.5cm} \today 
\vspace{1cm}
}
\end{center}

\section{Bounded variation and Lebesgue-Stieltjes integration for two-variate functions}\label{sec:integral}
\def\theequation{A.\arabic{equation}}
\setcounter{equation}{0}

\noindent
In this supplement, we briefly recapitulate some results on bounded variation and integration for two-variate functions. We begin by treating the case of functions defined on a compact rectangle in $\R^2$. Of particular interest is the  integration by parts formula in Theorem~\ref{theo:partialgen}. At the end of this appendix, we consider the case of potentially unbounded functions on open rectangles.

Let $A$ denote some rectangle in $\R^2$ and let $f$ be a real-valued function on $A$. For $\vect x, \vect y \in A$ such that $\vect x \le \vect y$ we set
\[
\Delta(f, x_1,x_2,y_1,y_2):= f(y_1, y_2) - f(x_1,y_2)-f(y_1,x_2) + f(x_1,x_2).
\]
For $\vect x, \vect y \in A$ such that $x_1<y_1$, we set
\[
\Delta_1(f, x_1, y_1;x_2):=f(y_1,x_2) - f(x_1,x_2)
\]
and finally, for $\vect x, \vect y \in A$ such that $x_2<y_2$, we set
\[
\Delta_2 (f, x_2,y_2;x_1):= f(x_1,y_2) - f(x_1,x_2).
\]
A function $f:A \to \R$ is called \textit{completely monotone} if $\Delta(f, x_1,x_2, y_1,y_2)\geq 0$ for any $\vect x, \vect y \in A$ such that $\vect x < \vect y$
, $\Delta_1(f ,x_1, y_1;x_2 )\geq 0$ for any $\vect x, \vect y \in A$ such that $x_1<y_1$ and $\Delta_2(f, x_2, y_2;x_1 )\geq 0$  for any $\vect x, \vect y \in A$ such that $x_2<y_2$.

\begin{definition}[Hardy-Krause variation]
Let $ \emptyset \ne [\vect a, \vect b]\subset \mathbb{R}^2$ and $f: [\vect a, \vect b] \to \R$. For $\vect x\in [\vect a, \vect b]$ and $\vect y\in [ \vect a, \vect b]$, we define
\begin{multline*}
VHK(f, [ \vect a , \vect x], \vect y) = \sup \sum_{i} \sum_j \vert  \Delta (f , s_i,t_j, s_{i+1},t_{j+1}) \vert \\
+ \sup \sum_{i} \vert \Delta_1(f, s_i, s_{i+1}; y_2) \vert + \sup \sum_{j} \vert \Delta_2(f,t_j, t_{j+1};y_1)\vert,
\end{multline*}
as the Hardy-Krause variation of $f$ on $[\vect a, \vect x]$ with anchor-point $\vect y$. Here, 
the supremum is taken over all decompositions $s_1, \dots , s_n$ and $t_1, \dots ,t_m$ with $a_1=s_1 < \dots < s_n =x_1$ and $a_2 = t_1< \dots< t_m=x_2$, respectively.
Let $BVHK([\vect a, \vect b])$ denote the space of all functions such that $VHK(f,[\vect a , \vect b ], \vect b) < \infty$. 
\end{definition}

\begin{example}
Consider $f:[0,1]^2 \to \R$, defined as $f(x,y) = \ind(x \ge 1/2, y \ge 1/2)$. 
Then $VHK(f, [ \vect 0,\vect 1], \vect y ) = 1$ for any  $\vect y \in [0,1/2)^2$, whereas $VHK(f, [ \vect 0,\vect 1], \vect y ) = 3$ for $\vect y \in [1/2,1]^2$.
\end{example}

The following simple properties are collected from \cite{Owe05} and \cite{AisDic14}:  
$BVHK([\vect a, \vect b])$ is closed under sums, differences and products. 
Moreover, for any $\vect x, \vect y\in [\vect a, \vect b]$, we have $VHK(f,[\vect a , \vect b ], \vect x) < \infty$ if and only if $VHK(f,[\vect a , \vect b ], \vect y) < \infty$, which can be derived directly from the definition. The following theorem is a refinement of Proposition 12 in \citealp{Owe05}.

\begin{theorem}[Theorem 2 in \citealp{AisDic14}] 
For any function $f \in BVHK([\vect a, \vect b])$ there exist unique functions $f^+,f^-: [\vect a, \vect b] \to \R$ such that (i), (ii) and (iii) hold:
\begin{enumerate}[(i)]
\item $f^+$ and $f^-$ are completely monotone
\item
$f^\pm(\vect a)=0$  and $f(\vect x) = f(\vect a) + f^+(\vect x) - f^-(\vect x)$ 
\item
$VHK(f,[\vect a, \vect b], \vect a) = VHK(f^+,[\vect a, \vect b], \vect a)  + VHK(f^-,[\vect a, \vect b], \vect a) $.
\end{enumerate}
\end{theorem}

The decomposition in \textit{(ii)} is called the \textit{Jordan decomposition} of $f$.
The explicit form of the functions $f^\pm$ is given in the proof of Theorem 2 of \cite{AisDic14} (see also \citealp{Har05}):
\begin{align*}
f^+(\vect x) &= \left\{  VHK(f, [\vect a, \vect x], \vect a) + f(\vect x) - f(\vect a) \right\} /2, \\ 
f^-(\vect x) &= \left\{  VHK(f, [\vect a, \vect x], \vect a) - f(\vect x) + f(\vect a) \right\} /2 .
\end{align*}

The next theorem shows that, if $f$ is additionally right-continuous, then it defines a unique signed measure on $[\vect a, \vect b]$. Also note that any signed measure $\nu$ on $\Bc([\vect a, \vect b])$ has a unique Jordan decomposition $\nu = \nu^+ - \nu^-$ with two measures $\nu^+$ and $\nu^-$, given by
\begin{align*}
\nu^+(A) &= \sup\{ \nu(B) : B \subset A, B \in\Bc([\vect a, \vect b]) \},  \\
\nu^-(A) &= - \inf\{ \nu(B) : B \subset A, B \in\Bc([\vect a, \vect b]) \}.
\end{align*}

\begin{theorem}[Theorem 3 in \citealp{AisDic14}]\label{theo:decomp2}
Let $f \in BVHK([\vect a, \vect b])$ be right-continuous. Then there exists a unique signed Borel-measure $\nu$ on $\Bc([\vect a, \vect b])$ such that 
\begin{align} \label{eq:fnu}
f(\vect x) = \nu([\vect a, \vect x]), \quad \vect x \in [\vect a, \vect b].
\end{align}
Moreover, if $f(\vect x) = f(\vect a) + f^+(\vect x) - f^-(\vect x)$ denotes the Jordan decomposition of $f$, and if $\nu=\nu^+ - \nu^-$ denotes the Jordan decomposition of $\nu$, then
\begin{align} \label{eq:fnu2}
f^+(\vect x) = \nu^+([\vect a, \vect x] \setminus \{ \vect a\} ), \qquad 
f^-(\vect x) = \nu^-([\vect a, \vect x] \setminus \{ \vect a\} )
\end{align}
for any $\vect x \in [\vect a, \vect b]$.
\end{theorem}
Note that, by \eqref{eq:fnu} and \eqref{eq:fnu2}, for any $\vect a \le \vect x < \vect y \le \vect b$,
\[
\nu((\vect x, \vect y]) = \Delta(f,x_1, x_2, y_1, y_2), \quad
\nu^\pm((\vect x, \vect y]) = \Delta(f^\pm,x_1, x_2, y_1, y_2).
\]

\begin{definition}
Let $f\in BVHK([\vect a, \vect b])$ be right-continuous. Let $g: [\vect a , \vect b] \to \R$ denote a measurable function such that either $\int_{[\vect a, \vect b]}  |g| \mathrm d\nu^+< \infty$ or $\int_{[\vect a, \vect b]}  |g| \mathrm d\nu^- < \infty$. Then
\[ 
\int_{[\vect a, \vect b]} g \mathrm{d}f := \int_{[\vect a, \vect b]} g \mathrm d\nu =\int_{[\vect a, \vect b]} g \mathrm d\nu^+ - \int_{[\vect a, \vect b]} g \mathrm d\nu^-
\]
denotes the Lebesgue-Stieltjes integral of $g$ with respect to $f$. 
\end{definition}

Given $f\in BVHK([\vect a, \vect b])$ and a fixed point $y \in [a_2, b_2]$, we can define a collection of functions $f_{1,y}:[a_1, b_1] \to \R$ through $f_{1,y}(x) = f(x,y)$. We have $f_{1,y} \in BV([a_1,b_1])$, hence, by a one-dimensional analog of the preceding developments, we obtain a unique Jordan decomposition
\[
f_{1,y} (x) = f_{1,y}(a_1) + f_{1,y}^+ (x) - f_{1,y}^- (x), \quad x \in [a_1,b_1]
\]
such that $f_{1,y}^\pm$ is non-decreasing with $f_{1,y}^\pm(a_1) = 0$ and $V(f_{1,y}, [a_1,b_1]) = V(f_{1,y}^+, [a_1,b_1])  + V(f_{1,y}^-, [a_1,b_1])$, where $V(f, [a_2, b_2])$ denotes the usual total variation of a real-valued function on a compact interval $[a_2, b_2]$. Attached is a unique signed measure $\nu_{1,y}$  such that
$
f_{1,y}(x) = \nu_{1,y}([a_1,x])$. Moreover, if $\nu_{1,y}=\nu_{1,y}^+ - \nu_{1,y}^-$ denotes the Jordan decomposition of $\nu_{1,y}$, then
\[
f_{1,y}^+(x) =  \nu_{1,y}^+((a_1,x] ) , \qquad  
f_{1,y}^-(x) =  \nu_{1,y}^-((a_1,x]  ) .  
\]
Note that the measure $\nu_{1,y}$ is related with $\nu$ by
\[
\nu_{1,y}([a_1,x])= \nu([a_1, x] \times [a_2, y]).
\]

The same arguments apply for the function $f_{2,x}:[a_2, b_2] \to \R$, defined through $f_{2,x}(y) = f(x,y)$, with $x\in [a_1, b_1]$ fixed. 
We will write $f(\mathrm dx, y)$ and $f(x, \mathrm dy)$ for $\nu_{1,y}(\mathrm dx)$ and $\nu_{2,x
}(\mathrm dy)$, respectively.

\begin{theorem}[Integration by parts] \label{theo:partialgen}
Let $\mu,\nu$ be finite signed measures on $[\vect a , \vect b]$ and, for $\vect x \in [\vect a , \vect b]$, write $f(\vect x) := \mu ([\vect a , \vect x])$, $g(\vect x) := \nu ([\vect a , \vect x])$.  Then, for any $(\vect c, \vect d] \subset [\vect a, \vect b]$ with $\vect c < \vect d$,
\begin{align*}
\int_{(\vect c , \vect d]} f \mathrm{d}g = & \int_{( \vect c, \vect d]} g \mathrm{d}f + \Delta(fg, c_1,c_2,d_1,d_2) \\
&- \int_{( c_1 ,d_1]} g(u ,d_2) f(\mathrm{d}u, d_2) + \int_{(c_1, d_1]}g(u ,c_2) f(\mathrm{d}u, c_2) \\
&- \int_{( c_2,d_2]} g(d_1,v) f(d_1,\mathrm{d}v) + \int_{(c_2, d_2]}g(c_1,v) f(c_1,\mathrm{d}v) \\
&+ \int_{( \vect c, \vect d]} \nu(\{u\}\times (v,d_2]) + \nu((u,d_1] \times\{v\}) + \nu(\{(u,v)\}) \mathrm{d}f(u,v)
\\
&+ \int_{( c_1 ,d_1]} \nu(\{u\}\times (c_2,d_2]) f(\mathrm{d}u, c_2)
\\
&+ \int_{( c_2 ,d_2]} \nu( (c_1,d_1]\times\{v\}) f(c_1,\mathrm{d}v) 
\end{align*}
\end{theorem}

Before we give a proof of this general result, we recall that to any right-continuous $f,g \in BVHK([\vect a , \vect b])$ there correspond unique signed measures $\mu,\nu$, respectively. If either $f$ or $g$ is additionally continuous, the last three terms in the representation above vanish and we obtain the following result

\begin{corollary}[Integration by parts] 
Let $f,g \in BVHK([\vect a , \vect b])$ be right-continuous functions with either $f$ or $g$ continuous.  Then, for any $(\vect c, \vect d] \subset [\vect a, \vect b]$,
\begin{multline*}
\int_{(\vect c , \vect d]} f \mathrm{d}g = \int_{( \vect c, \vect d]} g \mathrm{d}f + \Delta(fg, c_1,c_2,d_1,d_2) \\
- \int_{( c_1 ,d_1]} g(u ,d_2) f(\mathrm{d}u, d_2) + \int_{(c_1, d_1]}g(u ,c_2) f(\mathrm{d}u, c_2) \\
- \int_{( c_2,d_2]} g(d_1,v) f(d_1,\mathrm{d}v) + \int_{(c_2, d_2]}g(c_1,v) f(c_1,\mathrm{d}v). 
\end{multline*}

\end{corollary} 

\begin{proof}[Proof of Theorem \ref{theo:partialgen}]
First of all, we use the definition of $f,g$ and obtain
\begin{align*}
\int_{(\vect c, \vect d ]} f(u,v) \mathrm{d}g(u,v)= \int_{(\vect c, \vect d ]} \int_{(c_1,u]\times ( c_2 , v]}\mathrm{d}f(x,y) \mathrm{d}g(u,v) + R_1,
\end{align*}
where $R_1=\int_{( \vect c, \vect d ]} f(u,c_2)+f(c_1,v)-f(c_1, c_2) \mathrm{d}g(u,v)$.
Now, Fubini's Theorem yields
\begin{align*}
& \int_{(\vect c, \vect d ]} \int_{(c_1,u]\times ( c_2 , v]}\mathrm{d}f(x,y) \mathrm{d}g(u,v)
\\ 
=& \int_{(\vect c, \vect d ]} \int_{[x,d_1]\times [y,d_2]}\mathrm{d}g(u,v) \mathrm{d}f(x,y)
\\
=& \int_{(\vect c, \vect d ]} \int_{(x,d_1]\times (y,d_2]}\mathrm{d}g(u,v) \mathrm{d}f(x,y)
\\
&+ \int_{( \vect c, \vect d]} \nu(\{x\}\times (y,d_2]) + \nu((x,d_1] \times\{y\}) + \nu(\{(x,y)\}) \mathrm{d}f(x,y)
\end{align*}
and
\[
\int_{(\vect c, \vect d ]} \int_{(x,d_1]\times (y,d_2]}\mathrm{d}g(u,v) \mathrm{d}f(x,y)
\\
= \int_{( \vect c ,\vect d]} g \mathrm{d}f + R_2
\]
with $R_2=\int_{(\vect c ,\vect d]} g(d_1, d_2)- g(u,d_2)-g(d_1, v) \mathrm{d}f(u,v)$. Summarizing, \begin{multline*}
\int_{(\vect c, \vect d]} f \mathrm dg = \int_{(\vect c, \vect d]} g \mathrm df + R_1 + R_2\\
+ \int_{( \vect c, \vect d]} \nu(\{u\}\times (v,d_2]) + \nu((u,d_1] \times\{v\}) + \nu(\{(u,v)\}) \mathrm{d}f(u,v),
\end{multline*}
whence it remains to consider $R_1$ and $R_2$.
Observe that
\[
\int_{( \vect c, \vect d ]} f(u,c_2) \mathrm{d}g(u,v) 
= \int_{(c_1, d_1]} f(u,c_2)g(\mathrm{d}u,d_2)- \int_{(c_1, d_1]} f(u ,c_2 )g(\mathrm{d}u,c_2) 
\]
and a similar identity holds for $\int_{(\vect c ,\vect d]} g(u,d_2) \mathrm{d}f(u,v)$. To see this, note that the identity holds for $f(u,c_2) = I_{(\alpha,\beta]}(u)$ with $c_1\leq \alpha < \beta \leq d_1$ arbitrary; the general claim then follows by algebraic induction. Next, observe the following one-dimensional formula for integration by parts
\begin{align*}
\int_{(c_1, d_1]} f(u,c_2)g(\mathrm{d}u,d_2)  =& -\int_{(c_1, d_1]} g(u,d_2)f(\mathrm{d}u,c_2) 
\\
& + \int_{(c_1, d_1]} \nu(\{u\}\times[a_2,d_2]) f(\mathrm{d}u,c_2)
\\
& +g(d_1, d_2 )f(d_1,c_2) - g(c_1,d_2)f(c_1, c_2).
\end{align*}
This formula can be proved by an application of the Fubini Theorem after writing $f(u,c_2) = \int_{(c_1, u]} f(\mathrm{d}x,c_2) + f(c_1,c_2)$ and by observing that $\nu_{1,d_2}(\{u\}) = \nu(\{u\} \times [a_2,d_2])$. Thus we have (apply a similar formula for integration by parts to all four integrals below)
\begin{align*}
R_1=& \int_{(c_1, d_1]} f(u,c_2)g(\mathrm{d}u,d_2)- \int_{(c_1, d_1]} f(u ,c_2 )g(\mathrm{d}u,c_2) \\
& +\int_{(c_2,d_2]} f(c_1,v) g(d_1,\mathrm{d}v)- \int_{(c_2,d_2]} f(c_1,v) g(c_1,\mathrm{d}v)
\\
& - f(\vect c) \Delta(g,c_1,c_2,d_1,d_2) \\
=& -\int_{(c_1, d_1]} g(u,d_2)f(\mathrm{d}u,c_2) +g(d_1, d_2)f(d_1,c_2) - g(c_1,d_2)f(c_1, c_2)
\\
& + \int_{(c_1, d_1]} \nu(\{u\}\times[a_2,d_2]) f(\mathrm{d}u,c_2)
\\
& + \int_{(c_1, d_1]} g(u ,c_2 )f(\mathrm{d}u,c_2) - g(d_1,c_2)f(d_1,c_2)+g(c_1, c_2)f(c_1, c_2)\\
& - \int_{(c_1, d_1]} \nu(\{u\}\times[a_2,c_2]) f(\mathrm{d}u,c_2)
\\
&- \int_{(c_2,d_2]} g(d_1,v) f(c_1,\mathrm{d}v) + g(d_1, d_2) f(c_1,d_2) - g(d_1, c_2)f(c_1, c_2)\\
& + \int_{( c_2 ,d_2]} \nu( [a_1,d_1]\times\{v\}) f(c_1,\mathrm{d}v) 
\\
& + \int_{(c_2,d_2]} g(c_1,v) f(c_1,\mathrm{d}v)- g(c_1,d_2)f(c_1,d_2) + g(c_1,c_2) f(c_1, c_2)
\\
& - \int_{( c_2 ,d_2]} \nu( [a_1,c_1]\times\{v\}) f(c_1,\mathrm{d}v) 
\\
&  - f(\vect c) \Delta(g,c_1,c_2,d_1,d_2) .
\end{align*}
For $R_2$ we obtain
\begin{align*}
R_2=&  - \int_{(c_1, d_1]} g(u,d_2)f(\mathrm{d}u,d_2) + \int_{(c_1, d_1]} g(u,d_2 )f(\mathrm{d}u,c_2) \\
& - \int_{(c_2,d_2]} g(d_1,v) f(d_1,\mathrm{d}v) + \int_{(c_2,d_2]} g(d_1,v) f(c_1,\mathrm{d}v)
\\
& + g(d_1, d_2) \Delta(f,c_1,c_2,d_1,d_2).
\end{align*}
The result follows after collecting terms.
\end{proof}

\begin{definition}[Locally bounded variation and Lebesgue-Stieltjes integration] \label{def:locbvhk}
Consider $f:(\vect a, \vect b) \to \R$ which is potentially unbounded. We say that $f$ is of \textit{locally bounded Hardy-Krause variation}, notationally $f \in BVHK_{loc}((\vect a, \vect b))$ if and only if $f|_{[\vect c , \vect d]} \in BVHK([\vect c, \vect d])$ for any $\vect a < \vect c < \vect d < \vect b$. In the following, $f$ is assumed to be right-continuous. Let $\vect a_n, \vect b_n$ be two sequences converging to $\vect a$ and $\vect b$, respectively, and such that $\vect a < \vect a_{n+1}<\vect a_{n} < \vect b_n < \vect b_{n+1} < \vect b$. Since $f|_{[\vect a_n, \vect b_n]} \in BVHK([\vect a_n, \vect b_n])$, we can define unique measures $\nu_n^+$ and $\nu_n^-$ on $\Bc([\vect a_n, \vect b_n])$ as in Theorem~\ref{theo:decomp2}. Now, for $A \in \Bc((\vect a, \vect b))$ set
\[
\nu^\pm(A) := \lim_{n\to\infty} \nu_n^\pm(A \cap (\vect a_n, \vect b_n]).
\]
By monotone convergence, $\nu^+$ and $\nu^-$ are $[0,\infty]$-valued measures on $\Bc((\vect a, \vect b))$. Moreover, by Proposition~\ref{prop:restrict} below, the definition of $\nu^\pm$ is independent of the choice of the sequences $\vect a_n$ and $\vect b_n$. Finally, for $\vect a < \vect c< \vect d < \vect b$, the proposition implies that
\[
\nu((c, d]) := \nu^+((\vect c, \vect d ]) - \nu^-((\vect c, \vect d ])  = \Delta(f,c_1,c_2,d_1,d_2).
\]
Note that $\nu$ is not necessarily a signed measure on $\Bc((\vect a, \vect b))$, since expressions of the form ``$\infty-\infty$'' are possible in principle. Still, for a measurable function $g:(\vect a, \vect b) \to \R$ such that $\int |g | \mathrm d\nu^+<\infty$ or $\int |g| \mathrm d\nu^- < \infty$, we may define the Lebesgue-Stieltjes integral 
\[
\int_{(\vect a, \vect b)} g \mathrm df := \int_{(\vect a, \vect b)}  g \mathrm d\nu 
:= \int_{(\vect a, \vect b)}  g \mathrm d \nu^+ - \int_{(\vect a, \vect b)}  g \mathrm d\nu^-
\]
\end{definition}

\begin{prop} \label{prop:restrict}
Let $f \in BVHK([\vect a,\vect b])$ be right-continuous and let $\vect a < \vect c < \vect d < \vect b$. Set $g := f|_{[\vect c, \vect d]}$. Then, for any $A \in \Bc((\vect c, \vect d])$,
\[
\nu^\pm_f(A)=\nu^\pm_g(A),
\]
where $\nu^\pm_f$ and $\nu^\pm_g$ denote the unique measures associated to the unique signed measure $\nu_f$ of $f$ and $\nu_g$ of $g$, respectively.
\end{prop}
\begin{proof}
It suffices to show the identity on sets of the form $(\vect x, \vect y]\subset (\vect c, \vect d]$. By \eqref{eq:fnu}, we have
\[
\nu_f((\vect x,\vect y]) 
=
\Delta(f, x_1, x_2, y_1, y_2)
= \Delta(g, x_1, x_2, y_1, y_2)
=
\nu_g((\vect x, \vect y]).
\]
Uniqueness of the Jordan decomposition implies the assertion.
\end{proof}

\end{document}